\documentclass[12pt]{amsart}       %was 12pt
\makeatletter
\@namedef{subjclassname@2020}{\textup{2020} Mathematics Subject Classification}
\makeatother
\pdfoutput=1
\usepackage{subfig}
\usepackage{txfonts}
\usepackage [latin1]{inputenc}
\usepackage{amssymb}
\usepackage{eucal}
\usepackage{graphicx}
\usepackage{amsmath}
\usepackage{amscd}
\usepackage[all]{xy}           %xypic macro for latex2.09
\usepackage{tikz}
\usepackage{amsfonts,latexsym}
\usepackage{xspace}
\usepackage{epsfig}
\usepackage{float}
\usepackage{color}
\usepackage{fancybox}
\usepackage{colordvi}
\usepackage{multicol}
\usepackage{colordvi}
\usepackage{ifpdf}
\ifpdf
  \usepackage[colorlinks,final,backref=page,hyperindex]{hyperref}
\else
  \usepackage[colorlinks,final,backref=page,hyperindex,hypertex]{hyperref}
\fi
\usepackage[active]{srcltx} %SRC Specials for DVI Searching

\usepackage{graphicx}
\usepackage{epstopdf}
\usepackage{epsfig}

\def\a{\alpha}
\def\b{\beta}

\def\c{\cdot}

\def\D{\Delta}

\def\g{\gamma}

\def\o{\otimes}

\def\r{\rho}

\def\s{\sigma}

\def\v{\varepsilon}

\newtheorem{theorem}{Theorem}[section]
\newtheorem{prop}[theorem]{Proposition}
\theoremstyle{definition}
\newtheorem{defn}[theorem]{Definition}
\newtheorem{lemma}[theorem]{Lemma}
\newtheorem{coro}[theorem]{Corollary}
\newtheorem{prop-def}{Proposition-Definition}[section]
\newtheorem{coro-def}{Corollary-Definition}[section]

\newtheorem{remark}[theorem]{Remark}

\newtheorem{exam}[theorem]{Example}

\newcommand{\nc}{\newcommand}
\nc{\tred}[1]{\textcolor{red}{#1}}
\nc{\tblue}[1]{\textcolor{blue}{#1}}
\nc{\tgreen}[1]{\textcolor{green}{#1}}
\nc{\tpurple}[1]{\textcolor{purple}{#1}}
\nc{\btred}[1]{\textcolor{red}{\bf #1}}
\nc{\btblue}[1]{\textcolor{blue}{\bf #1}}
\nc{\btgreen}[1]{\textcolor{green}{\bf #1}}
\nc{\btpurple}[1]{\textcolor{purple}{\bf #1}}
\nc{\NN}{{\mathbb N}}
\nc{\ncsha}{{\mbox{\cyr X}^{\mathrm NC}}} \nc{\ncshao}{{\mbox{\cyr
X}^{\mathrm NC}_0}}

\nc{\bfk}{\mathbf{k}}
\newcommand{\efootnote}[1]{}
\newcommand{\delete}[1]{}

\nc{\mlabel}[1]{\label{#1}}  % Use this to suppress names
\nc{\mcite}[1]{\cite{#1}}  % Use this to suppress names
\nc{\mref}[1]{\ref{#1}}  % Use this to suppress names
\nc{\mbibitem}[1]{\bibitem{#1}} % Use this to show number name
\delete{
\nc{\mlabel}[1]{\label{#1}  % Use the next two lines to show names
{\hfill \hspace{1cm}{\small\tt{{\ }\hfill(#1)}}}}
\nc{\mcite}[1]{\cite{#1}{\small{\tt{{\ }(#1)}}}}  % Use this lines to show names
\nc{\mref}[1]{\ref{#1}{{\tt{{\ }(#1)}}}}  % Use this lines to show names
\nc{\mbibitem}[1]{\bibitem[\bf #1]{#1}} % Use this to show name
}

\nc{\li}[1]{\tblue{#1}}
\nc{\lir}[1]{\tblue{\underline{Li:} #1}}

\nc{\mclip}[2]{{\centering
		\includegraphics[scale=#1]{#2}
}}

\begin{document}

\title{Hopf Heap modules, Rota-Baxter Operators, and their structure theorems
}
%
%=========================================================================
\author{Huihui Zheng}
\address{School of Mathematics and Statistics, Henan Normal University, Xinxiang 453007, China}
\email{huihuizhengmail@126.com}

\author{Chan Zhao}
\address{School of Mathematics and Statistics, Henan Normal University, Xinxiang 453007, China}
\email{3227956863@qq.com}

\author{Liangyun Zhang$^\ast$}
\address{Nanjing Agricultural University, Nanjing 210095, China}
\email{zlyun@njau.edu.cn}
%========================================================================
\date{\today}
%========================================================================
\begin{abstract}
This paper is primarily devoted to the study of Hopf heaps and Hopf heap modules. We redefine the structure of Hopf trusses by means of Hopf heaps, establish the connection between Hopf trusses and Hopf braces, and provide a series of examples of Hopf truss structures from the perspective of Hopf heaps. Most importantly, we introduce the conception of Hopf heap modules, and present its structure theorem. Finally, we introduce the notions of Rota-Baxter operators on Hopf heaps and Hopf heap modules, and present the structure theorem for Rota-Baxter Hopf heap modules.

%In this paper, we are dedicated to the research of Hopf heap. We discussed the structure of Hopf heap and provided some of its properties. Through the structure of Hopf heap, we redefined Hopf truss. With this new definition, we established the connection between Hopf truss and Hopf brace, and provided numerous examples of Hopf truss from the perspective of Hopf heap. Most importantly, we defined the Hopf heap module and gave its stucture theorem. Finally, the definition of Rota-Baxter operators on Hopf heaps and Hopf heap modules were proposed and linked with the Rota-Baxter co-operators on commutative Hopf algebras and Hopf modules, respectively. Additionally, a new Hopf heap and Hopf heap module structure were constructed through the Rota-Baxter operators on Hopf heaps and Hopf heap modules, respectively.

\end{abstract}

\subjclass[2020]{16T05,16W99}
\keywords{Hopf algebra, Hopf heap, Hopf truss, Hopf heap module, Hopf module, Rota-Baxter operator.}

\maketitle

\tableofcontents

\setcounter{section}{0}

\allowdisplaybreaks

%========================================================================
\section{Introduction}

In the 1920s, Pr$\ddot{u}$fer \mcite{P} and Baer \mcite{Baer} introduced a notion called a heap. A heap is a set $H$ with a ternary operation
$[-,-,-]:H\times H\times H\rightarrow H$ satisfying the following axioms: for all $a,b,c,d,e\in H$,
$$
[[a,b,c],d,e]=[a,b,[c,d,e]],
$$
$$
[a,b,b]=a=[b,b,a].
$$

 There is a close relationship between heaps and groups. Any non-empty heap $H$ and $e\in H$, one can associate a group $G(H;e)=(H,[-,e,-])$, where $[-,e,-]$ is a binary operation acquired by fixing the middle variable in the ternary operation. Conversely, every group $(G,\c,e)$ can give a heap $H(G)=(G,[-,-,-])$, where $[x,y,z]=xy^{-1}z$, for all $x,y,z\in G$ (see \mcite{B1}).

 A skew left truss \mcite{B} is a set $A$ with binary operations $``\diamond"$ and $``\circ"$ and a function $\s:A\rightarrow A$ such that $(A,\diamond)$ is a group,  $(A,\circ)$ a semigroup, and the following axiom holds:
 $$
 a\circ (b\diamond c)=(a\circ b)\diamond\s(a)^{\diamond}\diamond (a\circ c),
 $$
 for all $a,b,c\in A$, where $\s(a)^{\diamond}$ denotes the inverse of $\s(a)$ in the group $(A,\diamond)$.

 This truss distributive law describes two different distributive laws: the well-known ring distribuvity and the one coming from the resently introduced braces. The former is obtained by setting $\s(a)=1_{\diamond}$, the latter is obtained by setting $\s$ to be the identity map. According to the connection between heaps and groups described above, the two structure rings and braces can be described elegantly by switching the group structure to a heap structure. This leads to the equivalent definition of a truss. A (two-sided) truss \mcite{B} is an abelian heap $T$ together with an associative binary operation $``\c"$ satisfying the conditions of ring and brace distributivity:
 \begin{center}
 $a\c [b,c,d]=[a\c b, a\c c, a\c d]$;
 \end{center}
  \begin{center} $[b,c,d]\c a=[b\c a,c\c a,d\c a], $
 \end{center}
for all $a,b,c,d\in A$. The notion of a truss can be weakened by not requesting that $(T,[-,-,-])$ is an abelian heap, in which case the system $(T, [-,-,-])$ is called a skew truss or near truss or not requesting the two-sided distributivity of $``\c"$ over $[-,-,-]$, in which case $(T, [-,-,-])$ is called a left or right skew truss.

The truss is concerned with the linearisation of heaps proposed by Grunspan in \mcite{Grunspan}, termed quantum cotorsors there and referred to as Hopf heaps in \mcite{BH}. Moreover, the authors in \mcite{BH} established an equivalence between the category of Hopf heaps and that of bi-Galois co-objects.

Rota-Baxter operators were initially introduced in Baxter's seminal work \mcite{Baxter} in 1960, in the framework of fluctuation theory. The theory of Rota-Baxter operators has been extensively developed across a wide range of mathematical disciplines. These operators are of fundamental importance, as they are closely linked to several central mathematical concepts, including the Yang-Baxter equation \mcite{BD,Schedler,Semenov}, dendriform algebras \mcite{Bai}, and the renormalization of quantum field theory \mcite{CK}. Rota-Baxter operators on groups were first introduced in \mcite{GLS}. The connections between Rota-Baxter groups, the Yang-Baxter equation, and skew braces were established in \mcite{BG}. As a group can be viewed as a fundamental example of a cocommutative Hopf algebra, the concept of Rota-Baxter operators on cocommutative Hopf algebras was naturally first proposed in \mcite{Goncharov}, and the relationship between such operators and Hopf braces was investigated in \mcite{ZZMG}.

The main purpose of this work is to further explore and investigate the structure of Hopf heaps, as well as their Rota-Baxter operators. As shown in \mcite{B}, every Hopf truss is a generalization of a Hopf brace. Naturally, leveraging the relationship between Hopf heaps and Hopf algebras (see Example 2.3 in \mcite{BH}), we can consider and study Hopf trusses from the perspective of Hopf heaps. Given that every Hopf algebra is a Hopf heap (a key aspect of their relationship), we introduce the concept of Rota-Baxter Hopf heap, establish the connection between Rota-Baxter Hopf heap and Rota-Baxter Hopf algebra. Furthermore, we define (Rota-Baxter) Hopf heap modules, establish the relationship between (Rota-Baxter) Hopf heap modules and (Rota-Baxter) Hopf modules, and present the structure theorem for (Rota-Baxter) Hopf heap modules.

We can summarize these relations in the following diagram--the black arrows represent the background of our research, while the red arrows denote the
main research content.

\vspace{4mm}
%TeXCAD Picture [333333333333333.pic]. Options:
%\grade{\on}
%\emlines{\off}
%\epic{\off}
%\beziermacro{\on}
%\reduce{\on}
%\snapping{\off}
%\quality{8.00}
%\graddiff{0.01}
%\snapasp{1}
%\zoom{4.0000}
\unitlength 1mm % = 2.85pt
\linethickness{0.4pt}
\ifx\plotpoint\undefined\newsavebox{\plotpoint}\fi % GNUPLOT compatibility
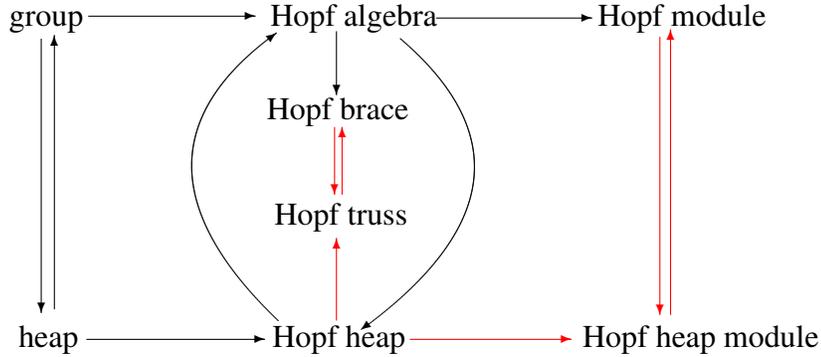
\begin{picture}(100,63)(0,0)
\put(17,28.25){heap}
\put(15.75,60){group}
\put(49,60){Hopf algebra}
\put(50,50.5){Hopf brace}
\put(51,38.5){Hopf truss}
\put(50.75,28.25){Hopf heap}
\put(92,60){Rota-Baxter Hopf algebra}
\put(92,28.25){\textcolor{red}{Rota-Baxter  Hopf heap}}
\put(20,58.25){\vector(0,-1){25.5}}
\put(21.75,33.25){\vector(0,1){25.5}}

\put(26.25,61.25){\vector(1,0){22.25}}
\put(26,29.5){\vector(1,0){23.75}}

\put(59.25,59.25){\vector(0,-1){6}}
\put(59,50.5){\textcolor{red}{\vector(0,-1){9}}}
\put(60,41.5){\textcolor{red}{\vector(0,1){9}}}
\put(59.25,30.75){\textcolor{red}{\vector(0,1){7}}}

\put(72.5,61){\vector(1,0){18.3}}
\put(68.8,29.25){\textcolor{red}{\vector(1,0){22.5}}}

\put(105.25,58.5){\textcolor{red}{\vector(0,-1){27}}}
\put(106.67,31.5){\textcolor{red}{\vector(0,1){28}}}
%\qbezvec(51.5,5.75)(28.63,27.63)(51.25,44)
\put(51.25,58){\vector(4,3){.07}}\qbezier(50,57)(26.63,42.63)(50.25,30)
%\end
%\qbezvec(63.5,5.25)(89.38,24.75)(67.75,43.25)
\put(63,30.4){\vector(-4,-3){.07}}\qbezier(64,31)(86.38,42.75)(65.75,58.25)

%\put(49,62){Hopf module}
%\put(90,62){\textcolor{red}{R-B Hopf module}}
%\put(45,-4.5){\textcolor{red}{Hopf heap module}}
%\put(90,-4.5){\textcolor{red}{R-B Hopf heap module}}
%\put(71.55,63.25){\textcolor{red}{\vector(1,0){18.25}}}
%\put(77,-3.5){\textcolor{red}{\vector(1,0){12.75}}}

%\put(59,13){\textcolor{red}{\vector(0,-1){15}}}
%\put(107,13){\textcolor{red}{\vector(0,-1){14}}}
%\put(60,48.5){\textcolor{red}{\vector(0,1){13}}}
%\put(105,48.5){\textcolor{red}{\vector(0,1){13}}}

%\end
%\put(105.25,63.5){\textcolor{red}{\vector(0,-1){27}}}
%\put(106.67,46.5){\textcolor{red}{\vector(0,1){28}}}

%\put(-4,-6){Hopf module}
%\put(90,60){\textcolor{red}{R-B Hopf module}}
%\put(45,30.5){\textcolor{red}{Hopf heap module}}
%\put(90,30.5){\textcolor{red}{R-B Hopf heap module}}
%\put(71.55,63.25){\textcolor{red}{\vector(1,0){18.25}}}
%\put(77,26.5){\textcolor{red}{\vector(1,0){12.75}}}

%TeXCAD Picture [333333333333333.pic]. Options:
%\grade{\on}
%\emlines{\off}
%\epic{\off}
%\beziermacro{\on}
%\reduce{\on}
%\snapping{\off}
%\quality{8.00}
%\graddiff{0.01}
%\snapasp{1}
%\zoom{4.0000}

\put(43,2.25){\textcolor{red}{Hopf heap module}}

\put(92,2.25){\textcolor{red}{Rota-Baxter  Hopf heap module}}

\put(59,27.25){\textcolor{red}{\vector(0,-1){22.5}}}
%\put(40.5,6.25){\textcolor{red}{\vector(0,1){27.5}}}

%\put(48,36){\textcolor{red}{\vector(1,0){32}}}
\put(75.8,3.25){\textcolor{red}{\vector(1,0){14.5}}}

\put(106.25,27.5){\textcolor{red}{\vector(0,-1){22}}}
%\put(95.67,6.5){\textcolor{red}{\vector(0,1){28}}}
%\put(49,62){Hopf module}
%\put(90,62){\textcolor{red}{R-B Hopf module}}
%\put(45,-4.5){\textcolor{red}{Hopf heap module}}
%\put(90,-4.5){\textcolor{red}{R-B Hopf heap module}}
%\put(71.55,63.25){\textcolor{red}{\vector(1,0){18.25}}}
%\put(77,-3.5){\textcolor{red}{\vector(1,0){12.75}}}

%\put(59,13){\textcolor{red}{\vector(0,-1){15}}}
%\put(107,13){\textcolor{red}{\vector(0,-1){14}}}
%\put(60,48.5){\textcolor{red}{\vector(0,1){13}}}
%\put(105,48.5){\textcolor{red}{\vector(0,1){13}}}

%\end
%\put(105.25,63.5){\textcolor{red}{\vector(0,-1){27}}}
%\put(106.67,46.5){\textcolor{red}{\vector(0,1){28}}}

%\put(-4,-6){Hopf module}
%\put(90,60){\textcolor{red}{R-B Hopf module}}
%\put(45,30.5){\textcolor{red}{Hopf heap module}}
%\put(90,30.5){\textcolor{red}{R-B Hopf heap module}}
%\put(71.55,63.25){\textcolor{red}{\vector(1,0){18.25}}}
%\put(77,26.5){\textcolor{red}{\vector(1,0){12.75}}}
\end{picture}

Overall, the paper is organized as follows. In Section \mref{sec:hp}, we recall the notion of a Hopf heap and give some of its properties.

In Section \mref{sec:hpht}, we give a equivalent definition of a Hopf truss by a Hopf heap. Using this equivalence definition of a Hopf truss, we construct some examples of Hopf trusses and Hopf heaps.

In Section \mref{sec:hpm}, we introduce the definition of Hopf heap modules and establish the structure theorem for such modules.

In section \mref{sec:rbohp}, we give the definition of a Rota-Baxter operator on a Hopf heap and construct a new Hopf heap structure via such an operator.

In Section \mref{sec:rbohpm}, a Rota-Baxter operator on a Hopf heap module is proposed, and its structure theorem is presented simultaneously.

{\bf Notations.} Throughout this paper, $\bfk$ is always considered in a fixed field. Unless otherwise specified, linearity, modules and tensor products are all taken over $\bfk$. And we freely use the Hopf algebra terminology introduced in \mcite{Sweedler}, and denote its  antipode of the Hopf algebra by $S$. For brevity, we write a comultiplication $\D(c)$ as $c_{1}\o c_{2}$ of a coalgebra $C$ without the summation sign, and denote its opposite coalgebra by $C^{cop}$, with the comultiplication $\D^{cop}(c)$ as $c_{2}\o c_{1}$, for $c\in C$. The set of group-like elements of a coalgebra $C$ is recorded as $G(C)$ and the assumption is made that $G(C)$ is always non-empty in this paper.

\section{Hopf heaps}
\mlabel{sec:hp}

In this section, we recall the definition of Hopf heaps and derive several results concerning Hopf heaps from the relationship between Hopf heaps and Hopf algebras.

\vspace{2mm}

\begin{defn}\mlabel{def:hopfheap}

 \mcite{BH} A Hopf heap is a coalgebra $C$ together with a coalgebra map
 $$
 \chi:C\o C^{co}\o C\rightarrow C, ~~a\o b\o c\mapsto [a,b,c]
 $$
 such that for all $a,b,c,d,h\in C$,
 \begin{eqnarray}
 &&[[a,b,c],d,h]=[a,b,[c,d,h]],\\
 &&[c_1,c_2,a]=[a,c_1,c_2]=\v(c)a.
 \end{eqnarray}
 The Hopf heap is denoted by $(C,\chi)$ or $(C,[-,-,-])$.
\end{defn}

\begin{remark}

\begin{enumerate}
\item A Hopf heap $C$ is said to be commutative, if $[x,y,z]=[z,y,x]$, for all $x,y,z\in C$.

\item Let $C$ and $D$ be Hopf heaps. A Hopf heap homorphism $f:C\rightarrow D$ is a coalgebra map such that $f([a,b,c])=[f(a),f(b),f(c)]$, for all $a,b,c\in C$. Denote the category of Hopf heaps by $\mathbf{HP}$.

 \item Let $(C,[---])$ be a cocommutative Hopf heap (that is, it is a cocommutative coalgebra as a Hopf heap). Then $(C,[---]^{op})$ is also a Hopf heap, where $[x,y,z]^{op}=[z,y,x]$, for all $x,y,z\in C$.
\end{enumerate}
\end{remark}

\begin{exam}\mlabel{exam:hopfheap1}
\begin{enumerate}
  \item Let $C$ be a vector space with basis $\{u,\theta\}$. Define
$$
\D(u)=u\o \theta+ \theta \o u, ~~\D(\theta)=\theta\o \theta-u\o u,
$$
$$\v(u)=0, ~~\v(\theta)=1.
$$

Then $(C,\D, \v)$ is a coalgebra.

Define
$$
[u,u,u]=-u, [\theta,\theta,\theta]=\theta, ~\text{else is zero}.
$$

Then it easy to prove that $(C,[-,-,-])$ is a commutative Hopf heap.

 \item Let $(C, \D_{C}, [-,-,-]_{C})$ and $(D, \D_{D}, [-,-,-]_{D})$ be two Hopf heaps. Then $C\o D$ is a Hopf heap with the following structures
$$\D_{C\o D}(a\o h)=a_{1}\o h_{1}\o a_{2}\o h_{2},$$
$$\v_{C\o D}(a\o h)=\v_{C}(a)\v_{D}(h),$$
$$[a\o h,b\o g,c\o f]_{C\o D}=[a,b,c]_{C}\o [h,g,f]_{D},$$
for any $a,b,c\in C; g,h,f\in D$.

%It is easy to see $(C\o D,\varDelta_{C\o D})$ is a coalgebra.
%\begin{eqnarray*}
	%&&[a\o h,b\o g,c\o f]_{1}\o [a\o h,b\o g,c\o f]_{2}\\
	%&=& ([a,b,c]_{C}\o [h,g,f]_{D})_{1}\o ([a,b,c]_{C}\o [h,g,f]_{D})_{2} \\
	%&=& [a_{1},b_{2},c_{1}]_{C}\o [h_{1},g_{2},f_{1}]_{D}\o [a_{2},b_{1},c_{2}]_{C}\o [h_{2},g_{1},f_{2}]_{D}\\
	%&=& [a_{1}\o h_{1},b_{2}\o g_{2},c_{1}\o f_{1}]_{C\o D}\o [a_{2}\o h_{2},b_{1}\o g_{1},c_{2}\o f_{2}]_{C\o D}\\
	%&=& [(a\o h)_{1},(b\o g)_{2},(c\o f)_{1}]_{C\o D}\o [(a\o h)_{2},(b\o g)_{1},(c\o f)_{2}]_{C\o D}
%\end{eqnarray*}
%\begin{eqnarray*}
	%[[a\o h,b\o g,c\o f]_{C\o D},d\o m,e\o n]_{C\o D}
	%&=& [[a,b,c]_{C}\o [h,g,f]_{D},d\o m,e\o n]_{C\o D}\\
	%&=& [[a,b,c]_{C},d,e]_{C}\o [[h,g,f]_{D},m,n]_{D}\\
	%&=& [a,b,[c,d,e]_{C}]_{C}\o [h,g,[f,m,n]_{D}]_{D}\\
%&=& [a\o h,b\o g,[c,d,e]_{C}\o [f,m,n]_{D}]_{C\o D}\\
	%	&=& [a\o h,b\o g,[c\o f,d\o m,e\o n]_{C\o D}]_{C\o D}\\
%\end{eqnarray*}
%\begin{eqnarray*}
	%[(a\o h)_{1},(a\o h)_{2},b\o g]_{C\o D}
	%&=& [a_{1}\o h_{1},a_{2}\o h_{2},b\o g]_{C\o D}\\
	%&=& [a_{1},a_{2},b]_{C}\o [h_{1},h_{2},g]_{D}\\
	%&=& \varepsilon(a)b\o \varepsilon(h)g\\
	%&=&  \varepsilon(a\o h)b\o g
%\end{eqnarray*}
%Similarly, $[b\o g,(a\o h)_{1},(a\o h)_{2}]_{C\o D}=\varepsilon(a\o h)b\o g$.
%Then, $ (C\o D,\varDelta_{C\o D}, [-,-,-]_{C\o D} )$ is a Hopf heap.
\end{enumerate}
\end{exam}

\begin{exam} \mcite{BH}\mlabel{exam:hopfheaphopfalg}
Let $H$ be a Hopf algebra with antipode $S$. Then $H$ is a Hopf heap with the operation:
$$[a,b,c]=aS(b)c,$$
 for all $a,b,c\in H$. This Hopf heap associated to the Hopf algebra $H$ is denoted by $Hp(H,[-,-,-])$.

Conversely, if $(C,\chi)$ is a Hopf heap, then, for any $x\in G(C)$, the coalgebra $C$ can be made into a Hopf algebra with identity $x$, whose multiplication and antipode are given by
$$
ab=[a,x,b], ~S(a)=[x,a,x],
$$
for all $a,b\in H.$

In what follows, we denote the introduced Hopf algebra by $H_x(C)$.

\end{exam}

\begin{remark} Let $\mathbf{HA}$ be a category of Hopf algebras. Then there is a functor $F: \mathbf{HA} \rightarrow \mathbf{HP}$ from the category of Hopf algebras
to the category of Hopf heaps.

In fact, the functor $F:\mathbf{HA} \rightarrow \mathbf{HP}$ can be given by
$$
F(H)=Hp(H,[-,-,-]), ~~F(f)=f,
$$
where the first homomorphism $f$ is a Hopf algebra map from $H$ to $G$.

\end{remark}

\begin{remark}\mlabel{rem:HaHp}
\begin{enumerate}
 \item Let $(C,\chi)$ be a Hopf heap, $x\in G(C)$. One easily checks that the Hopf heap associated to the Hopf algebra $H_x(C)$ is equal to $C$, that is, $Hp(H_x(C),[-,-,-]_x)=C$.

In fact, we have
$$
[a,b,c]_x=aS(b)c=a[x,b,x]c=[a,x,[x,b,x]]c=[a,b,x]c=[[a,b,x],x,c]=[a,b,c].,
$$for any $a,b,c\in C$.

 \item Let $H$ be a Hopf algebra with antipode $S$ and unit $e$. Then $H_e(Hp(H,[-,-,-]))=H$.

Indeed, by Example \mref{exam:hopfheaphopfalg}, we know that $Hp(H,[-,-,-])$ is a Hopf heap with $[a,b,c]=aS(b)c$, for all $a,b,c\in H$, and $H_e(Hp(H,[-,-,-]))$ is a Hopf algebra with
$$a\c_e b=[a,e,b]=aS(e)b=ab,$$
$$S_e(a)=[e,a,e]=eS(a)e=S(a).$$
Hence $H=H_e(H,[-,-,-])$.

\end{enumerate}
\end{remark}

\begin{defn}\mcite{BH}
  A Grunspan map for a Hopf heap $C$ is a coalgebra homomorphism $\vartheta: C \rightarrow C$, such that
	$$ [[a, b, \vartheta(c)], d, h] = [a, [d, c, b], h],$$
for all $a, b, c, d, h \in C$.
\end{defn}

\begin{remark}\mlabel{rem:HAHG}\mcite{BH}
Every Hopf heap admits the Grunspan map. The Grunspan map is given by
$$\vartheta : C \rightarrow C, c \mapsto\sum [c_{1}, [f_{1}, c_{3}, c_{2}],f_{2}],$$
	where $f \in C$ is a given element with $\v(f) = 1$.

\end{remark}

\begin{prop}
Let $C$ be a commutative Hopf heap. Then
\begin{eqnarray}
	&&[[a,b,c],d,h]=[a,[b,c,d],h]=[a,b,[c,d,h]],
\end{eqnarray}
for all $a,b,c,d,h\in C$.
\end{prop}
\begin{proof}
	If $C$ is commutative, then, by Remark \mref{rem:HAHG}, we get
	\begin{eqnarray*}
		\vartheta(c)
		&=&[c_{1}, [e_{1}, f_{3}, c_{2}],f_{2}]=[c_{1}, [c_{2}, c_{3}, f_{1}],f_{2}]=[c,  f_{1},f_{2}]=\v(f)c=c,
	\end{eqnarray*}for any $c\in C$, where $\v(f) = 1$.
	Furthermore,  we have
	$$ [[a, b, c], d, h]=[[a, b, \vartheta(c)], d, h] = [a, [d, c, b], h]=[a, [b, c, d], h],$$
for all $a,b,c,d,h\in C$.
	\end{proof}

In the following, we present a sufficient and necessary condition for a Hopf heap homomorphism to be a Hopf algebra homomorphism.

\begin{prop}\mlabel{prop:hahphomomo}
Let $H$ and $G$ be Hopf heaps, and $x\in G(H)$ and $x'\in G(G)$. Then $f:H\rightarrow G$ is a Hopf heap homomorphism and $f(x)=x'$, if and only if $f:H_x(H)\rightarrow H_{x'}(G)$ is a Hopf algebra homomorphism.

In particular, $f:H\rightarrow G$ is a Hopf heap homomorphism, if and only if $f:H_x(H)\rightarrow H_{f(x)}(G)$ is a Hopf algebra homomorphism.
\end{prop}

\begin{proof}
It is easy to see that $f:H_x(H)\rightarrow H_{x'}(G)$ is a Hopf algebra homomorphism, if and only if
$$f([a,x,b])=[f(a),x',f(b)],$$ for all $a,b\in H$, and $f$ is a coalgebra map.

Suppose that $f:H\rightarrow G$ is a Hopf heap homomorphism and $f(x)=x'$. Then $f$ is a coalgebra map, and
 $$f([a,x,b])=[f(a),f(x),f(b)]=[f(a),x',f(b)],$$
 for all $a,b\in H$,
 that is, $f:H_x(H)\rightarrow H_{x'}(G)$ is an algebra map. So $f:H\rightarrow G$ is a Hopf heap homomorphism.

Conversely, if $f:H_x(H)\rightarrow H_{x'}(G)$ is a Hopf algebra homomorphism, then $f(x)=x'$, and for any $a,b,c\in H$,
\begin{eqnarray*}
&&f([a,b,c]_x)=f(aS_H(b)c)=f(a)f(S_H(b))f(c)=f(a)S_G(f(b))f(c)=[f(a),f(b),f(c)]_{x'}.
\end{eqnarray*}

Hence $f:Hp(H_x(H),[-,-,-]_x)\rightarrow Hp(H_{x'}(G),[-,-,-]_{x'})$ is a Hopf heap homomorphism.

By Remark \mref{rem:HaHp}, we know that $Hp(H_x(H),[-,-,-]_x)=H$ and $Hp(H_{x'}(G),[-,-,-]_{x'})=G$, so $f:H\rightarrow G$ is a Hopf heap homomorphism.
\end{proof}

\begin{lemma}\mlabel{lem:Hppro}
Let $H$ be a commutative Hopf heap. Then the following equality is satisfied:
\begin{eqnarray}
&&[[w,w',w''],[y,y',y''],[z,z',z'']]=[[w,y,z],[w',y',z'],[w'',y'',z'']],
\end{eqnarray}
for all $w,y,z,w',y',z',w'',y'',z''\in H$.
\end{lemma}

\begin{proof} It is easy to see that $H_x(H)$ is a commutative Hopf algebra with $x\in G(H)$. Moreover we have
\begin{eqnarray*}
[[w,w',w''],[y,y',y''],[z,z',z'']]&=&[[w,w',w'']_x,[y,y',y'']_x,[z,z',z'']_x]_x\\
&=&wS(w')w''S(y'')y'S(y)zS(z')z''\\
&=&wS(y)zS(w')y'S(z')w''S(y'')z''\\
&=&[w,y,z]_x[S(w'),S(y'),S(z')]_x[w'',y'',z'']_x\\
&=&[w,y,z]_xS[w',y',z']_x[w'',y'',z'']_x\\
&=&[[w,y,z]_x,[w',y',z']_x,[w'',y'',z'']_x]_x\\
&=&[[w,y,z],[w',y',z'],[w'',y'',z'']],
\end{eqnarray*}
for any $w,y,z,w',y',z',w'',y'',z''\in H$.
\end{proof}

\begin{defn}\mcite{BH}
Let $(C,\chi)$ be a Hopf heap. We call the linear map
$$
\tau^b_{a}:C\rightarrow C, ~c\mapsto \chi(c\o a\o b)=[c,a,b],
$$
a right $(a,b)$-translation, for $a,b\in C$.

Symmetrically, we call the linear map
$$
\s^b_{a}:C\rightarrow C, ~c\mapsto \chi(a\o b\o c)=[a,b,c],
$$
a left $(a,b)$-translation, for  $a,b\in C$.

\end{defn}

\begin{remark} %\mcite{BH}

Let $(C,\chi)$ be a Hopf heap. Then, we can obtain the following equalities:
$$\tau^d_c\circ \tau^b_a=\tau^{[b,c,d]}_a,$$
$$\sigma^b_a\circ\sigma^d_c=\sigma_{[a,b,c]}^d,$$
for all $a,b,c,d\in C$.
\end{remark}

\begin{lemma}\mlabel{lem:cohp}
Let $H$ be a Hopf heap, $G$ a commutative Hopf heap, and let $f:H\rightarrow G$ a coalgebra map. Then $f:H\rightarrow G$ is a Hopf heap homomorphism, if and only if $\tau^{x'}_{f(x)}\circ f:H\rightarrow G$ is a Hopf heap homomorphism, for all $x\in G(H)$, and $x'\in G(G)$.
\end{lemma}

\begin{proof}
Since $f$ is a coalgebra map, $f(x)\in G(G)$.

 By Lemma \mref{lem:Hppro}, we can prove that $\tau^{x'}_{f(x)}: G\rightarrow G$ is a Hopf heap homomorphism: for all $a,b,c\in G$, \begin{eqnarray*}
 [\tau^{x'}_{f(x)}(a),\tau^{x'}_{f(x)}(b),\tau^{x'}_{f(x)}(c)] &=&[[a,f(x),x'],[b,f(x),x'],[c,f(x),x']]\\
 &=&[[a,b,c],[f(x),f(x),f(x)],[x',x',x']]\\
 &=&[[a,b,c],f(x),x']\\
 &=&\tau^{x'}_{f(x)}([a,b,c]),
 \end{eqnarray*}
and
 \begin{eqnarray*}
 \D(\tau^{x'}_{f(x)}(a))&=&\D([a,f(x),x'])\\
 &=&[a_1,f(x),x']\o [a_2,f(x),x']\\
 &=&\tau^{x'}_{f(x)}(a_1)\o \tau^{x'}_{f(x)}(a_2).
 \end{eqnarray*}

 Hence $\tau^{x'}_{f(x)}\circ f:H\rightarrow G$ is a Hopf heap homomorphism, if and only if $\tau^{x'}_{f(x)}\circ f$ is a coalgebra map and
 $$(\tau^{x'}_{f(x)}\circ f)([h,h',h''])=[\tau^{x'}_{f(x)}\circ f(h),\tau^{x'}_{f(x)}\circ f(h'),\tau^{x'}_{f(x)}\circ f(h'')]=\tau^{x'}_{f(x)}([f(h), f(h'),f(h'')]),$$
 for any $h,h',h''\in H,$
 that is, $$\tau^{x'}_{f(x)} (f[h,h',h''])=\tau^{x'}_{f(x)}([f(h), f(h'),f(h'')]).$$

Suppose that $\tau^{x'}_{f(x)}\circ f:H\rightarrow G$ is a Hopf heap homomorphism. Then we have
$$\tau^{f(x)}_{x'}\circ \tau^{x'}_{f(x)}(f[h,h',h''])=\tau^{f(x)}_{x'}\circ \tau^{x'}_{f(x)}([f(h), f(h'),f(h'')]).$$

 Moreover
 $$\tau^{f(x)}_{x'}\circ \tau^{x'}_{f(x)}(a)=\tau^{[x',x',f(x)]}_{f(x)}(a)
 =\tau^{f(x)}_{f(x)}(a)=[a,f(x),f(x)]=a,
 $$
 for any $a\in H$, so $f([h,h',h''])=[f(h), f(h'),f(h'')]$,  that is, $f:H\rightarrow G$ is a Hopf heap homomorphism.

 Conversely, if $f:H\rightarrow G$ is a Hopf heap homomorphism, that is, for any $h,h',h''\in H,$
  $$f([h,h',h''])=[f(h), f(h'),f(h'')],$$
then we get
 $$\tau^{x'}_{f(x)} (f[h,h',h''])=\tau^{x'}_{f(x)}([f(h), f(h'),f(h'')]),$$
 so, by the above proof, we know that $\tau^{x'}_{f(x)}\circ f:H\rightarrow G$ is a Hopf heap homomorphism.
\end{proof}

\begin{prop}\mlabel{coro:cohp}
Let $H$ be a Hopf heap, $G$ a commutative Hopf heap, and let $f:H\rightarrow G$ be a coalgebra map. Then $f:H\rightarrow G$ is a Hopf heap homomorphism if and only if $\tau^{x'}_{f(x)}\circ f:H_x(H)\rightarrow H_{x'}(G)$ is a Hopf algebra homomorphism.

Here $x\in G(H)$ and $x'\in G(G)$.

\end{prop}
\begin{proof}
Since $f$ is a coalgebra map, we know that
\begin{center}
$f(x)\in G(G),$
\end{center}
\begin{center}
$(\tau^{x'}_{f(x)}\circ f)(x)=\tau^{x'}_{f(x)}(f(x))=[f(x),f(x),x']=x'.$
\end{center}

 By Proposition \mref{prop:hahphomomo}, we know that $\tau^{x'}_{f(x)}\circ f:H_x(H)\rightarrow H_{x'}(G)$ is a Hopf algebra homomorphism if and only if $\tau^{x'}_{f(x)}\circ f:H\rightarrow G$ is a Hopf heap homomorphism.

 Again by Lemma \mref{lem:cohp}, $f:H\rightarrow G$ is a Hopf heap homomorphism if and only if $\tau^{x'}_{f(x)}\circ f:H\rightarrow G$ is a Hopf heap homomorphism.

 Hence $f:H\rightarrow G$ is a Hopf heap homomorphism if and only if $\tau^{x'}_{f(x)}\circ f:H_x(H)\rightarrow H_{x'}(G)$ is a Hopf algebra homomorphism.
\end{proof}

\begin{remark}
There is a functor $F' : \mathbf{HP} \rightarrow \mathbf{HA}$ from the category of Hopf heaps
to the category of Hopf algebras.

In fact, the functor $F':\mathbf{HP} \rightarrow \mathbf{HA}$ is given by
$$
F'(H)=H_x(H), ~~F'(g)=\tau^{x'}_{g(x)}\circ g,
$$
where $x\in G(H)$, $x'\in G(G)$ and the homomorphism $g$ is a Hopf heap map from $H$ to $G$.

\end{remark}

\begin{coro}
Let $(G,\D_G,1_G,S_G)$ be a Hopf algebra, $(H,\D_H,1_H,S_H)$ a commutative Hopf algebra, and $f:G\rightarrow H$ a coalgebra map. Then $f: Hp(G,[-,-,-]_G)\rightarrow Hp(H,[-,-,-]_H)$ is a Hopf heap homomorphism, if and only if $\tau^{1_H}_{f(1_G)}\circ f:G\rightarrow H$ is a Hopf algebra homomorphism, where $(\tau^{1_H}_{f(1_G)}\circ f)(x)=f(x)S_H(f(1_G))$, for all $x\in G$.

\end{coro}

\begin{proof}
By Proposition \mref{coro:cohp}, we know that $f: Hp(G,[-,-,-]_G)\rightarrow Hp(H,[-,-,-]_H)$ is a Hopf heap homomorphism, if and only if $\tau^{1_H}_{f(1_G)}\circ f:H_{1_G}(Hp(G,[-,-,-]_G))\rightarrow H_{1_H}(Hp(H,[-,-,-]_H))$ is a Hopf algebra homomorphism.

Again by the facts that $H_{1_G}(Hp(G,[-,-,-]_G))=G$ and $H_{1_H}(Hp(H,[-,-,-]_H))=H$, we obtain that $\tau^{1_H}_{f(1_G)}\circ f$ is a Hopf algebra homomorphism from $G$ to $H$.
\end{proof}

\section{Hopf heaps and Hopf trusses}\mlabel{sec:hpht}

In this section, we build a connection between Hopf heaps and Hopf trusses. Furthermore, we give a equivalent definition of a Hopf truss by a Hopf heap, and give some constructions of Hopf trusses by Hopf heaps.

\begin{defn}\mcite{B}\mlabel{def:hopftruss} Let $(H,\D,\v)$ be a coalgebra. Assume that there exist two binary operations $``\c"$ and $``\circ"$ on the coalgebra $(H,\D,\v)$  such that $(H,\c,\D,\v)$ is a Hopf algebra (with unit $1$ and antipode $S$), and  $(H,\circ,\D,\v)$ a nonunital bialgebra. We say that $H$ is a (left) Hopf truss if there exists a coalgebra endomorphism $\sigma$ of $(H,\D,\v)$ such that
\begin{eqnarray}
 &&a\circ (bc)=(a_1\circ b)S(\s(a_2))(a_3\circ c),
\end{eqnarray}
for all $a,b,c\in H$.
\end{defn}

In what follows, we denote the Hopf truss $H$ by $(H,\cdot,\circ,\sigma)$, and call the coalgebra endomorphism $\sigma$ a cocyle.

\begin{remark}
It is obvious that every Hopf brace in the sense of \mcite{AGV} is a Hopf truss with the cocycle $\s=$ id.
\end{remark}

\begin{prop} \mlabel{prop:hopftrussandheap}
%If $(H,\chi)$ is a Hopf heap, and $(H,\circ,\D,\v)$ is a bialgebra. For any $a,b,c,d\in H$, $a\circ [b,c,d]=[a_1\circ b, a_2\circ c,a_3\circ d]$, then $(H,\c_e,\circ), e\in G(C)$ is a Hopf truss.

%Conversely, if $(H,\c,\circ,\s)$ is a Hopf truss, then $H$ is a Hopf heap and $a\circ [b,c,d]=[a_1\circ b, a_2\circ c,a_3\circ d]$, for all $a,b,c\in H$.

Let $(H,\chi)$ be a Hopf heap, and $(H,\circ,\D,\v)$ a nonunital bialgebra. Then, $H$ is a Hopf truss if and only if
\begin{eqnarray}
&&a\circ [b,c,d]=[a_1\circ b, a_2\circ c,a_3\circ d],
\end{eqnarray}
 for any $a,b,c,d\in H$.
\end{prop}

\begin{proof}
Since $(H,\chi)$ is a Hopf heap, by Example \mref{exam:hopfheaphopfalg}, $H_x(H)$ is a Hopf algebra with multiplication and antipode
$$a\c_x b=[a,x,b],\ \ S(a)=[x,a,x],$$
for any $a,b\in H$, where $x\in G(H).$

Suppose that the condition (6) holds, and $\s_x(a)=a\circ x$. Then $\s_x$ a coalgebra endomorphism such that $(H,\c_x,\circ,\sigma_x)$ is a Hopf truss: for all $a,b,d\in H$,
\begin{eqnarray*}
 (a_1\circ b)\c_x S(\s_x(a_2))\c_x(a_3\circ d)&=&(a_1\circ b)\c_x S(a_2\circ x)\c_x(a_3\circ  d)\\
 &=&(a_1\circ b)\c_x [x,a_2\circ x,x]\c_x(a_3\circ d)\\
 &=&[a_1\circ b,x,[x,a_2\circ x,x]]\c_x(a_3\circ d)\\
 &=&[a_1\circ b,a_2\circ x,x]\c_x(a_3\circ d)\\
 &=&[[a_1\circ b,a_2\circ x,x],x,a_3\circ d]\\
 &=&[a_1\circ b, a_2\circ x,a_3\circ d]\\
 &=&a\circ [b,x,d]\\
 &=&a\circ (b\c_x d).
\end{eqnarray*}

Conversely, if $(H,\c,\circ,\s)$ is a Hopf truss, we can define
$$[a,b,c]=aS(b)c,
 $$
 for any $a,b,c\in H$, then $H$ is a Hopf heap. Again by Theorem 6.4 in \mcite{B}: $(1)\Longleftrightarrow (5)$, we get
 $$a\circ [b,c,d]=[a_1\circ b, a_2\circ c,a_3\circ d],$$
  for all $a,b,c,d\in H$.
\end{proof}

By Proposition \mref{prop:hopftrussandheap}, we get a equivalent definition of a Hopf truss.

\begin{defn}\mlabel{def:newhopftruss}

  Let $(H,\chi=[-,-,-])$ be a Hopf heap, and $(H,\circ,\D,\v)$ a nonunital bialgebra. We call $H$ is a Hopf truss if
 $$
 a\circ [b,c,d]=[a_1\circ b, a_2\circ c,a_3\circ d],
 $$
for any $a,b,c,d\in H$.
\end{defn}

Now, by Definition \mref{def:newhopftruss}, we can establish the relation between Hopf trusses and Hopf braces.

\begin{prop} (a) Every Hopf truss $(H,\chi,\circ)$ in which $(H,\circ)$ is a Hopf algebra (with unit 1) gives rise to a Hopf brace.

(b) Every Hopf brace $(H,\c,\circ)$ gives rise to a Hopf truss.
\end{prop}

\begin{proof}
(a) By Example \mref{exam:hopfheaphopfalg}, if we define $ab=[a,1,b], ~S(a)=[1,a,1]$, for $a,b\in H$, then, we easily prove that $(H, \c,S)$ is a Hopf algebra with unit 1.

Now we prove that $(H,\c,\circ)$ is a Hopf brace: for any $a,b,c\in H$, we have
\begin{eqnarray*}
a\circ (bc)&=&a\circ [b, 1, c]\\
&=&[a_1\circ b, a_2, a_3\circ c]\\
&=&[[a_1\circ b,1,1], a_2, [1,1,a_3\circ c]]\\
&=&[a_1\circ b,1,[1, a_2, [1,1,a_3\circ c]]]\\
&=&[a_1\circ b,1,[[1, a_2, 1],1,a_3\circ c]]\\
&=&[a_1\circ b,1,[S(a_2),1,a_3\circ c]]\\
&=&[a_1\circ b,1,S(a_2)(a_3\circ c)]\\
&=&(a_1\circ b)S(a_2)(a_3\circ c).
\end{eqnarray*}

(b) Define
$$[a,b,c]=aS(b)c,$$
 for any $a,b,c\in H$. Then by Example \mref{exam:hopfheaphopfalg}, $(H,[-,-,-])$ is a Hopf heap.

Now we prove that $H$ is a Hopf truss by Lemma 1.7 in \mcite{AGV}: for any $a,b,c,d\in H$, we have
\begin{eqnarray*}
a\circ [b,c,d]&=&a\circ (bS(c)d)\\
&=&(a_1\circ b)S(a_2)(a_3\circ(S(c)d))\\
&=&(a_1\circ b)\underbrace{S(a_2)(a_3\circ S(c))}S(a_4)(a_5\circ d)\\
&=&(a_1\circ b)S(a_2\circ c)a_3S(a_4)(a_5\circ d)\\
&=&(a_1\circ b)S(a_2\circ c)(a_3\circ d)\\
&=&[a_1\circ b, a_2\circ c,a_3\circ d],
\end{eqnarray*}
so, by Definition \mref{def:newhopftruss}, $H$ is a Hopf truss.
\end{proof}

In what follows, we construct a new Hopf truss by a Hopf truss.

\begin{prop}

 Assume $(H,\c,\circ,\s)$ is a Hopf truss. Then, for any $x\in G(H)$, we know that $(H,\c_x,\circ,\s_x)$ is a Hopf truss, whose new multiplication $``\c_x"$ and new operator $``\s_x"$ on $H$ are given by
 $$a\c_x b=aS(x)b,\ \ \s_x(a)=a\circ x,$$
for any $a,b\in H.$
\end{prop}

\begin{proof}
Since $(H,\c,\D,\v,S)$ is a Hopf algebra, $(H,[-,-,-])$ is a Hopf heap with $[a,b,c]=aS(b)c$, for $a,b,c\in H$. Hence $a\c_x b=aS(x)b=[a,x,b]$, such that $(H,\c_x)=H_x(H,[-,-,-])$ is a Hopf algebra, and by Proposition \mref{prop:hopftrussandheap} we have
\begin{eqnarray*}
a\circ (b\c_x c)&=&a\circ[b,x,c]\\
&=&[a_1\circ b,a_2\circ x,a_3\circ c]\\
&=&[a_1\circ b,\s_x(a_2),a_3\circ c]\\
&=&(a_1\circ b)S(\s_x(a_2))(a_3\circ c).
\end{eqnarray*}

 Hence $(H,\c_e,\circ)$ is a Hopf truss.

\end{proof}

In the following, we find that the cocycle in a Hopf truss $(H,\c,\circ,\s)$ for its unit $1_{\c}$ being in the central $Z(H,\circ)$ is a Hopf heap endomorphism.

\begin{prop} Let $(H,\c,\circ,\s)$ be a Hopf truss such that its unit $1_{\c}$ is in the central $Z(H,\circ)$. Then $\s$ is an endomorphism of the induced Hopf heap $(H,[-,-,-])$.
\end{prop}

\begin{proof}By Lemma 6.2 in \mcite{B}, we know that $\s(a)=a\circ 1_{\c}$, so, for all $a,b,c\in H$,
\begin{eqnarray*}
\s([a,b,c])&=&[a,b,c]\circ 1_{\c}\\
&=&1_{\c}\circ [a,b,c]\\
&=&[1_{\c}\circ a,1_{\c}\circ b,1_{\c}\circ c]\\
&=&[a \circ 1_{\c},b \circ 1_{\c},c \circ 1_{\c}]\\
&=&[\s(a),\s(b),\s(c)],\\
\D(\s(a))&=&\D(a\circ 1_{\c})=a_1\circ 1_{\c}\o a_2\circ 1_{\c}\\
&=&\s(a_1)\o \s(a_2).
\end{eqnarray*}
\end{proof}

In what follows, we present a number of examples of Hopf trusses arising from Hopf heaps.

\begin{lemma} Let $(H,[-,-,-])$ be a Hopf heap, and $x\in G(H)$. Define a multiplication
$$a\c_x b=\varepsilon(a)\varepsilon(b)x,\ \ a,b\in H.
$$

 Then $(H,\c_x,[-,-,-])$ is a Hopf truss.
\end{lemma}

\begin{proof}
It is obvious that $(H,\c_x,\D)$ is a bialgebra. Moreover we have
 $$a\c_x [b,c,d]=[a_1\c_x b, a_2\c_x c,a_3\c_x d],$$
  for all $a,b,c,d\in H$. So $(H,\c_x,[-,-,-])$ is a Hopf truss by Definition \mref{def:newhopftruss}.
\end{proof}

\begin{lemma} \mlabel{lem:constructhopftruss} Let $(H,[-,-,-])$ be a commutative and cocommutative Hopf heap, and $\a$ an idempotent Hopf heap endomorphism of $H$. Define two binary operations as follows:
$$
x\c_{\a} y=[x_1,\a(x_2),y],$$
$$x\c'_{\a} y=[x,\a(y_1),y_2],
$$
for all $x,y\in H$. Then $(H,[-,-,-],\c_{\a})$ and $(H,[-,-,-],\c'_{\a})$ are Hopf trusses.
\end{lemma}

\begin{proof}
Firstly, we prove that the associativity of multiplication $``\c_{\a}"$ is satisfied: for all $x,y,z\in H$,
\begin{eqnarray*}
(x\c_{\a} y)\c_{\a} z&=&[x_1,\a(x_2),y]\c_{\a} z\\
&=&[[x_1,\a(x_2),y]_1,\a([x_1,\a(x_2),y]_2),z]\\
&=&[[x_{1},\a(x_{4}),y_1],[\a(x_{2}),\a(x_{3}),\a(y_2)],z]\\
&=&[[x_{1},\a(x_{3}),y_1],[\a(x_{2})_1,\a(x_{2})_2,\a(y_2)],z]\\
&=&[[x_{1},\a(x_{2}),y_1],\a(y_2),z]\\
&=&[x_{1},\a(x_{2}),[y_1,\a(y_2),z]]\\
&=&x\c_{\a} (y\c_{\a} z).
\end{eqnarray*}

Secondly, we prove that $\D$ is an algebra map:
\begin{eqnarray*}
\D(x\c_{\a} y)&=&\D([x_1,\a(x_2),y])\\
&=&[x_1,\a(x_2),y]_1\o [x_1,\a(x_2),y]_2\\
&=&[x_1,\a(x_4),y_1]\o [x_2,\a(x_3),y_2]\\
&=&[x_1,\a(x_2),y_1]\o [x_3,\a(x_4),y_2]\\
&=&x_1 \c_{\a} y_1\o x_2 \c_{\a} y_2.
\end{eqnarray*}

Finally, by Lemma \mref{lem:Hppro}, we can prove: for all $w,x,y,z\in H$,
\begin{eqnarray*}
[w_1\c_{\a} x, w_2\c_{\a} y, w_3\c_{\a} z]&=&[[w_{11},\a(w_{12}),x],[w_{21},\a(w_{22}),y],[w_{31},\a(w_{32}),z]]\\
&=&[[w_{11},w_{21},w_{31}],[\a(w_{12}),\a(w_{22}),\a(w_{32})],[x,y,z]]\\
&=&[[w_{1},w_{2},w_{3}],[\a(w_{4}),\a(w_{5}),\a(w_{6})],[x,y,z]]\\
&=&[w_{1},\a(w_{2}),[x,y,z]]\\
&=&w\c_{\a}[x,y,z].
\end{eqnarray*}

Thus, according to Proposition \mref{prop:hopftrussandheap}, $(H,[-,-,-],\c_{\a})$ is a Hopf truss.

Similarly, we can prove that $(H,\chi,\c'_{\a})$ is also a Hopf truss.
\end{proof}

\begin{coro} Any a commutative and cocommutative Hopf heap $(H,[-,-,-])$ with either of multiplications:
$$
xy=\v(x)y\ ; \ \ xy=\v(y)x, x,y\in H,
$$
is a Hopf truss.
\end{coro}

\begin{proof}
It is straightforward by setting $\a=$ id in Lemma \mref{lem:constructhopftruss}.
\end{proof}

\begin{coro} Let $(H,[-,-,-])$ be a commutative and cocommutative Hopf heap, and $x\in G(H)$. Then $H$ together with the multiplication:
$$
ab=[a,x,b], \ \ a,b\in H,
$$
is a Hopf truss.
\end{coro}

\begin{proof}
It is straightforward by setting $\a(a)=\varepsilon(a)x$ in Lemma \mref{lem:constructhopftruss}, for all $a\in H$.
\end{proof}

In the following, we find that the set of all endomorphisms of a commutative and cocommutative Hopf heap has a Hopf truss structure.

\begin{prop}
Let $(H,[-,-,-])$ be a commutative and cocommutative Hopf heap. Then the set $E(H)$ of all endomorphisms of commutative Hopf heap $H$ has a Hopf truss structure, whose multiplication $``\circ"$ is the composition map, comultiplication  defined by $\widetilde{\D}(\a)=\a\o \a$, and heap operation $``\{-,-,-\}"$ given by
$$
\{\a,\b,\g\}:H\rightarrow H, ~x\mapsto [\a(x_1),\b(x_2),\g(x_3)],
$$
for all $\a,\b,\gamma\in E(H)$.
\end{prop}

\begin{proof}
Firstly, we prove that the operation $``\{-,-,-\}"$ is well-defined, that is, we need to prove $\{\a,\b,\g\}\in E(H)$.

As a matter of fact, for all $\a,\b,\g\in E(H)$, and $x,y,z\in H$, we have
\begin{eqnarray*}
\{\a,\b,\g\}([x,y,z])&=&[\a[x,y,z]_1,\b[x,y,z]_2,\g[x,y,z]_3]\\
&=&[\a[x_1,y_3,z_1],\b[x_2,y_2,z_2],\g[x_3,y_1,z_3]]\\
&=&[[\a(x_1),\a(y_3),\a(z_1)],[\b(x_2),\b(y_2),\b(z_2)],[\g(x_3),\g(y_1),\g(z_3)]]\\
&=&[[\a(x_1),\b(x_2),\g(x_3)],[\a(y_3),\b(y_2),\g(y_1)],[\a(z_1),\b(z_2),\g(z_3)]]\\
&=&[\{\a,\b,\g\}(x),\{\a,\b,\g\}(y),\{\a,\b,\g\}(z)].
\end{eqnarray*}

Secondly, we prove that $E(H)$ has a Hopf truss structure.

It is easy to prove that $\D(\{\a,\b,\g\}(x))=\{\a,\b,\g\}(x_1)\o \{\a,\b,\g\}(x_2)$, so $\{\a,\b,\g\}$ is a coalgebra map.

Moreover, we easily prove that $(E(H), \circ,\widetilde{\D})$ is a bialgebra, $(E(H), \{-,-,-\},\widetilde{\D})$ a Hopf heap, and for any $f,\a,\b,\g\in E(H)$, we have
\begin{eqnarray*}
(f\circ \{\a,\b,\g\})(x)&=&f(\{\a,\b,\g\}(x))\\
&=&f([\a(x_1),\b(x_2),\g(x_3)])\\
&=&[f\circ\a(x_1),f\circ\b(x_2),f\circ\g(x_3)]\\
&=&[f_1\circ\a(x_1),f_2\circ\b(x_2),f_3\circ\g(x_3)]\\
&=&\{f_1\circ \a,f_2\circ\b,f_3\circ\g\}(x).
\end{eqnarray*}

So, by Proposition \mref{prop:hopftrussandheap}, we know that $E(H)$ has a Hopf truss structure.

\end{proof}

\vspace{2mm}

\section{Hopf heap modules and their structure theorems}
\mlabel{sec:hpm}

In this section, we introduce the conception of Hopf heap modules, and give its structure theorem.

\begin{defn}\mlabel{def:hopfheapmodule1}

 \begin{enumerate}
\item Let $(H,\chi)$ be a Hopf heap, and $(M,\r)$ a left $H$-comodule (denote $\r(m)=m_{(-1)}\o m_{(0)}$). We call $M$ a left $H$-Hopf heap module if there is an action $\rhd: H\o H^{co}\o M\rightarrow M$ by $(c\o d)\rhd m$, such that for any $a,b,c,d\in H, m\in M$,
 \begin{eqnarray}
 &&([a,b,c]\o d)\rhd m=(a\o b)\rhd ((c\o d)\rhd m),\\
 &&(c_1\o c_2)\rhd m=\v(c)m,\\
 &&\r((c\o d)\rhd m)=[c_1,d_2,m_{(-1)}]\o (c_2\o d_1)\rhd m_{(0)}.
\end{eqnarray}
In the following, we denote the $H$-Hopf heap module $M$ by $(M,\rhd,\r)$.

 \item Let $(M,\rhd_1,\r_1)$ and $(N,\rhd_2,\r_2)$ be two left $H$-Hopf heap modules, and $f: M\rightarrow N$ a linear map. We call $f$ a left $H$-Hopf heap module map if it is a left $H$-comodule map and
$$f((a\o b)\rhd_1m)=(a\o b)\rhd_2f(m),
$$ for any $a,b\in H, m\in M$.

Denote the category of left $H$-Hopf heap modules by $\mathbf{HPM}$.
\end{enumerate}
\end{defn}

\begin{exam}\mlabel{ex:hopfheapmodule}
\begin{enumerate}
\item Let $(H,\chi)$ be a Hopf heap. Then $(H,\rhd\equiv\chi,\D)$ is a left $H$-Hopf heap module.
\\
In particular, if $H$ is a Hopf algebra with antipode $S$, then, by Example \mref{exam:hopfheaphopfalg}, we have a
Hopf heap $(H,[-,-,-])$. So, we get a left $H$-Hopf heap module $(H,\rhd,\D)$, where $(h\otimes g)\rhd l=hS(g)l$ for $h,g,l\in H$.

 \item Let $(H,\chi)$ be a Hopf heap, and $M$ a vector space.
Define two linear maps:
\begin{center}
$\rhd: H\o H^{co}\o H\o M\rightarrow H\o M$, $(h\o g)\rhd (l\o m)=[h,g,l]\o m$,
\end{center}
 \begin{center}
$\r:H\o M\rightarrow H\o H\o M, \r(h\o m)=h_1\o h_2\o m$.
\end{center}
Then $(H\o M,\rhd, \r)$ is a left $H$-Hopf heap module.
 \end{enumerate}

 \begin{proof} (a) The proof is straightforward.

 (b) It is obvious that $(H\o M, \r)$ is a left $H$-comodule. Moreover, we have
 \begin{eqnarray*}
 ([a,b,c]\o d)\rhd (g\o m)&=&[[a,b,c],d,g]\o m\\
 &=&[a,b,[c,d,g]]\o m\\
 &=&(a\o b)\rhd ((c\o d)\rhd (g\o m)),\\
 (c_1\o c_2)\rhd (g\o m)&=&[c_1,c_2,g]\o m=\v(c)g\o m,\\
 \r((c\o d)\rhd (g\o m))&=&\r([c,d,g]\o m)\\
 &=&[c,d,g]_1\o [c,d,g]_2\o m\\
 &=&[c_1,d_2,g_1]\o [c_2,d_1,g_2]\o m\\
 &=&[c_1,d_2,g_1]\o (c_2\o d_1)\rhd (g_2\o m)\\
 &=&[c_1,d_2,(g\o m)_{(-1)}]\o (c_2\o d_1)\rhd (g\o m)_{(0)},
 \end{eqnarray*}
  for any $a,b,c,d,g\in H$ and $m\in M$.
  \end{proof}
\end{exam}

\begin{defn}\mlabel{def:hopfmodule}

 \mcite{Sweedler} Let $H$ be a Hopf algebra, $(M,\cdot)$ a left $H$-module, and $(M,\r)$ a left $H$-comodule. If for any $h\in H, m\in M$,
 $$
 \r(h\c m)=h_1m_{(-1)}\o h_2\c m_{(0)},
 $$
 we call $M$ a left $H$-Hopf module.

 Denote the category of left $H$-Hopf modules by $\mathbf{HM}$.
\end{defn}

In the following, we establish the relations between Hopf heap modules and Hopf modules.

\begin{prop}\mlabel{prop:hopfmodulehopfheapmodule}
\begin{enumerate}
\item Let $M$ be a left $H$-Hopf heap module. Then $M$ is a left $H_x(H)$-Hopf module.

 \item Conversely, if $M$ is a left $H$-Hopf module, then $M$ is a left $Hp(H,[-,-,-])$-Hopf heap module.
     \end{enumerate}
\end{prop}

\begin{proof} (a) Since $(H,\chi)$ is a Hopf heap, $H_x(H)$ is a Hopf algebra with identity $x$, where $x\in G(H)$.

Define the action as follows:
$$h\c m=(h\o x)\rhd m, h\in H, m\in M.$$

Then, for any $h,g\in H, m\in M$, we have
\begin{eqnarray*}
 h\c(g\c m)&=&(h\o x)\rhd ((g\o x)\rhd m)=([h,x,g]\o x)\rhd m\\
 &=&(hg\o x)\rhd m=(hg)\c m,\\
 x\c m&=&(x\o x)\rhd m=m,\\
 \r(h\c m)&=&\r((h\o x)\rhd m)=[h_1,x,m_{(-1)}]\o (h_2\o x)\rhd m_{(0)}\\
 &=&h_1m_{(-1)}\o h_2\c m_{(0)},
 \end{eqnarray*}
so, $M$ is a left $H_x(H)$-Hopf module.

(b) The Hopf algebra $H$ implies that $(H,[-,-,-])$ is a Hopf heap by setting $[a,b,c]=aS(b)c$ for $a,b,c\in H$.

Define the action: $(a\o b)\rhd m=(aS(b))\c m$, for $a,b\in H$ and $m\in M$. We have
\begin{eqnarray*}
(c_1\o c_2)\rhd m&=&(c_1S(c_2))\c m=\v(c)m,\\
(a\o b)\rhd ((c\o d)\rhd m)&=&(a\o b)\rhd ((cS(d))\c m)\\
&=&(aS(b))\c ((cS(d))\c m)\\
&=&(aS(b)cS(d))\c m\\
&=&(aS(b)c\o d)\rhd m\\
&=&([a,b,c]\o d)\rhd m,\\
\r((c\o d)\rhd m)&=&\r((cS(d))\c m)\\
&=&(cS(d))_1m_{(-1)}\o (cS(d))_2\c m_{(0)}\\
&=&c_1S(d_2)m_{(-1)}\o (c_2S(d_1))\c m_{(0)}\\
&=&[c_1,d_2,m_{(-1)}]\o (c_2\o d_1)\rhd m_{(0)},
\end{eqnarray*}
for any $a,b,c,d\in H$. Thus, by Definition \mref{def:hopfheapmodule1}, $M$ is a left $Hp(H,[-,-,-])$-Hopf heap module.
\end{proof}

\begin{remark}
\begin{enumerate}
\item Suppose that $f:M\rightarrow N$ is a left $H$-Hopf heap module map. Then $f$ is a left $H_{x}(H)$-Hopf module map.

In fact, by Proposition \mref{prop:hopfmodulehopfheapmodule}, $M$ and $N$ are left $H_{x}(H)$-Hopf modules. In addition, for any $h\in H, m\in M$,
$$
f(h\c m)=f((h\o x)\rhd m)=(h\o x)\rhd f(m)=h\c f(m).
$$

Hence, there is a functor $G:\mathbf{HPM}\rightarrow \mathbf{HM}$ from the category of left $H$-Hopf heap modules to the category of left $H_{x}(H)$-Hopf modules. The functor is given by
$$
G(M)=M,\ G(f)=f.
$$

 \item Suppose that $f:M\rightarrow N$ is a left $H$-Hopf module map. Then $f$ is a left $Hp(H,[-,-,-])$-Hopf heap module map.

In fact, by Proposition \mref{prop:hopfmodulehopfheapmodule}, $M$ and $N$ are left $Hp(H,[-,-,-])$-Hopf heap modules, and
$$f((a\o b)\triangleright_M m)=f((aS(b))\c m)=(aS(b))\c f(m)=(a\o b)\triangleright_N f(m),$$
for any $a,b\in H$ and $m\in M$, that is, $f$ is a left $Hp(H,[-,-,-])$-Hopf heap module map.

Hence, there is a functor $F:\mathbf{HM}\rightarrow \mathbf{HPM}$ from the category of left $H$-Hopf modules to the category of left $Hp(H,[-,-,-])$-Hopf heap modules. The functor is given by
$$
F(M)=M,\ F(f)=f.
$$

 \item Let $H$ be a Hopf algebra, and $G(H)=\{e\}$. Then there is an isomorphism of categories as follows:
$$\mathbf{HM}\stackrel{F}\cong \mathbf{HPM}.
$$

In fact, by Remark \mref{rem:HaHp}, we know that
$$H_e(Hp(H,[-,-,-]))=H,$$
thus, according to items (a) and (b), we obtain the above isomorphism.
\end{enumerate}
\end{remark}

In what follows, we give the structure theorem of Hopf heap modules.

\begin{theorem}\mlabel{the:thestructuretheorem}
Let $(M,\rhd,\r)$ be a left $H$-Hopf heap module, and $M^{coH}_x=\{m\in M, \r(m)=x\o m\}$ where $x\in G(H)$. Then there exists an isomorphism of $H$-Hopf heap modules:
$$M\cong H\o M^{coH}_x.$$
\end{theorem}

\begin{proof}
By Example \mref{ex:hopfheapmodule}, we know that $H\o M^{coH}_x$ is a Hopf heap module. Define a linear map
\begin{center}
$P:M\rightarrow M$ by $P(m)=(x\o m_{(-1)})\rhd m_{(0)}$.
%$P(m)=([x,m_{(-1)},x]\o x)\rhd m_{(0)}$.
\end{center}

In what follows, we prove that $ImP\subset M^{coH}_x$.

In fact, for any $m\in M$, we get
\begin{eqnarray*}
\r(P(m))&=&\r((x\o m_{(-1)})\rhd m_{(0)})\\
&=&[x,m_{(-1)2},m_{(0)(-1)}]\o (x\o m_{(-1)1})\rhd m_{(0)(0)}\\
&=&[x,m_{(-1)2},m_{(-1)3}]\o (x\o m_{(-1)1})\rhd m_{(0)}\\
&=&x\o ((x\o m_{(-1)})\rhd m_{(0)})\\
&=&x\o P(m).
\end{eqnarray*}

Define two maps as follows:
$$\a:H\o M^{coH}_x\rightarrow M,\  \a(a\o m)=(a\o x)\rhd m,
 $$
 $$\b:M\rightarrow H\o M^{coH}_x, \ \b(m)=m_{(-1)}\o P(m_{(0)}).$$

  Nextly, we prove that $\a\b=$ id and $\b\a=$ id.

  In fact, for any $m\in M,$ we have
\begin{eqnarray*}
\a\b(m)&=&\a(m_{(-1)}\o P(m_{(0)}))\\
&=&\a(m_{(-1)}\o ((x\o m_{(0)(-1)})\rhd m_{(0)(0)}))\\
&=&\a(m_{(-1)1}\o ((x\o m_{(-1)2})\rhd m_{(0)}))\\
&=&(m_{(-1)1}\o x)\rhd (((x\o m_{(-1)2})\rhd m_{(0)}))\\
&=&(m_{(-1)1}\o m_{(-1)2})\rhd m_{(0)}\\
&=&m,
\end{eqnarray*}
and hence for any $a\in H, n\in M^{coH}_x$, we obtain
\begin{eqnarray*}
\b\a(a\o n)&=&\b((a\o x)\rhd n)\\
&=&((a\o x)\rhd n)_{(-1)}\o P(((a\o x)\rhd n)_{(0)})\\
&=&[a_1, x, n_{(-1)}]\o P((a_2\o x)\rhd n_{(0)})\\
&=&[a_1, x, x]\o P((a_2\o x)\rhd n)\\
&=&a_1\o P((a_2\o x)\rhd n)\\
&=&a_1\o ((x\o((a_2\o x)\rhd n)_{(-1)})\rhd ((a_2\o x)\rhd n)_{(0)})\\
&=&a_1\o ((x\o [a_{21},x,n_{(-1)}])\rhd ((a_{22}\o x)\rhd n_{(0)}))\\
&=&a_1\o ((x\o [a_{2},x,x])\rhd ((a_{3}\o x)\rhd n))\\
&=&a_1\o ((x\o a_{2})\rhd ((a_{3}\o x)\rhd n))\\
&=&a_1\o ([x,a_{2}, a_{3}]\o x)\rhd n\\
&=&a\o ((x\o x)\rhd n)\\
&=&a\o n.
\end{eqnarray*}

Finally, we prove that $\a$ is a Hopf heap module map: for any $a\in H, n\in M^{coH}_x$, we get
\begin{eqnarray*}
\r\a(a\o n)&=&\r((a\o x)\rhd n)\\
&=&[a_1,x,n_{(-1)}]\o ((a_2\o x)\rhd n_{(0)})\\
&=&[a_1,x,x]\o ((a_2\o x)\rhd n)\\
&=&a_1\o ((a_2\o x)\rhd n)\\
&=&(\mathrm{id}\o\a)(a_1\o a_2\o n)\\
&=&(\mathrm{id}\o\a)\r(a\o n),
\end{eqnarray*}
and for any $a,b,g\in H$, we have
\begin{eqnarray*}
\a((a\o b)\rhd_{H\o M^{coH}}(g\o n))&=&\a([a,b,g]\o n)\\
&=&([a,b,g]\o x)\rhd n\\
&=&(a\o b)\rhd (g\o x)\rhd n\\
&=&(a\o b)\rhd\a((g\o n).
\end{eqnarray*}
\end{proof}

\begin{remark}
(a) Let $(M,\rhd,\r)$ be a left $H$-Hopf heap module. Then by Proposition \mref{prop:hopfmodulehopfheapmodule}, $M$ is a left $H_{x}(H)$-Hopf module, where $H_{x}(H)$ is a Hopf algebra with unit $x$. So, there exists an isomorphism $H_{x}(H)$-Hopf modules:
$$M\cong H_{x}(H)\o M^{coH}.$$

(b) Let $M$ be a left $H$-Hopf module. Then by Proposition \mref{prop:hopfmodulehopfheapmodule}, $M$ is a left $Hp(H,[-,-,-])$-Hopf heap module. So, by Theorem \mref{the:thestructuretheorem}, for any given element $x\in G(H)$, there exists an isomorphism $H$-Hopf heap modules:
$$M\cong Hp(H,[-,-,-])\o M^{coH}_x.$$

\end{remark}

\section{Rota-Baxter operators on Hopf heaps}\mlabel{sec:rbohp}

In this section, we introduce the definition of Rota-Baxter Hopf heaps, and build some connections between Rota-Baxter operators on Hopf heaps and Rota-Baxter co-operators on Hopf algebras.

\subsection{Rota-Baxter operators on Hopf heaps}
\begin{defn}\mlabel{def:rbohopfheap}
 Let $(C, \D, \chi)$ be a Hopf heap. A linear map $B:C\rightarrow C$ is called a Rota-Baxter operator on $(C, \D, \chi)$ if for all $a, b, c\in C:$
 \begin{eqnarray}
 &&B([a,b,c])=[B(a),B(b),B(c)],\\
&&B(a_{1})\o B(a_{2})=[B(a)_{1},B(B(a)_{4}),B(B(a)_{2})]\o B(a)_{3}.
\end{eqnarray}
\end{defn}

\begin{exam}\mlabel{exam:rbohopfheap5}
\begin{enumerate}
\item Let $C$ be a vector space with a basis $\{u,\theta\}$. Define
$$
\D(u)=u\o \theta+ \theta \o u, ~~\D(\theta)=\theta\o \theta-u\o u,
$$
$$\v(u)=0, ~~\v(\theta)=1.
$$
$$
[u,u,u]=-u, [\theta,\theta,\theta]=\theta, ~\text{else is zero}.
$$

By Example \mref{exam:hopfheap1}, we know that $(C,\D,\chi)$ is a commutative Hopf heap. Hence, by Definition \mref{def:rbohopfheap}, we have two kinds of nontrivial Rota-Baxter operators on the Hopf heap $(C, \D, \chi)$ given by

(i)~$B(u)=0, B(\theta)=\theta;$

(ii)~$B(u)=-u, B(\theta)=\theta.$

Now let's prove that ~(ii)~is valid:
$$B([\theta,\theta,\theta])=\theta=[B(\theta),B(\theta),B(\theta)],$$
$$B([u,u,u])=u=[B(u),B(u),B(u)],$$
\begin{eqnarray*}
	\D^{4}(u)
	&=& u\o \theta\o \theta\o \theta+\theta\o u\o \theta\o \theta+\theta\o \theta\o u\o \theta + \theta\o \theta\o \theta\o u\\
	&=&-u\o u\o u\o \theta-u\o u\o \theta\o u-u\o \theta\o u\o u-\theta\o u\o u\o u,
\end{eqnarray*}
\begin{eqnarray*}
	\D^{4}(\theta)
	&=& \theta\o \theta\o \theta\o \theta-u\o u\o \theta\o \theta-u\o \theta\o u\o \theta - \theta\o u\o u\o \theta\\
	&=&-u\o \theta\o \theta\o u-\theta\o u\o \theta\o u-\theta\o \theta\o u\o u+u\o u\o u\o u,
\end{eqnarray*}
\begin{eqnarray*}
	[B(u)_{1},B(B(u)_{4}),B(B(u)_{2})]\o B(u)_{3}&=&-[u_{1},B(u_{4}),B(u_{2})]\o u_{3}\\
	&=&-\theta\o u-u\o \theta\\
    &=&B(\theta)\o B(u)+B(u)\o B(\theta)\\
    &=&B(u_{1})\o B(u_{2}),
\end{eqnarray*}
\begin{eqnarray*}
	[B(\theta)_{1},B(B(\theta)_{4}),B(B(\theta)_{2})]\o B(\theta)_{3}&=& [\theta_{1},B(\theta_{4}),B(\theta_{2})]\o \theta_{3}\\
		&=&\theta\o \theta-u\o u\\
	&=&B(\theta)\o B(\theta)-B(u)\o B(u)\\
	&=&B(\theta_{1})\o B(\theta_{2}).
\end{eqnarray*}

 Similarly, we can prove that given map $B$ with $B(u)=0$ and $B(\theta)=\theta$ is also a nontrivial Rota-Baxter operator on the Hopf heap $(C,\D,\chi)$.

 \item Let $B_{1}$ be a Rota-Baxter operator on $(C, \D_{C}, \chi_{C})$, and $B_{2}$ a Rota-Baxter operator on $(D, \D_{D}, \chi_{D})$.
Then, by Example \mref{exam:hopfheap1}, $(C\o D, \D_{C\o D}, \chi_{C\o D})$ is a Hopf heap.

Define $B:C\o D\rightarrow C\o D,c\o h\mapsto B_{1}(c)\o B_{2}(h)$. Then $B$ is a Rota-Baxter operator on the Hopf heap $(C\o D, \D_{C\o D}, \chi_{C\o D})$.

In fact, for any $a,b,c\in C, h,g,f\in D$, we have
\begin{eqnarray*}
	B([a\o h,b\o g,c\o f]_{C\o D})
	&=& B([a,b,c]_{C}\o [h,g,f]_{D})\\
	&=& B_{1}([a,b,c]_{C})\o B_{2}([h,g,f]_{D})\\
	&=& [B_{1}(a),B_{1}(b),B_{1}(c)]_{C}\o [B_{2}(h),B_{2}(g),B_{2}(f)]_{D}\\
	&=&  [B_{1}(a)\o B_{2}(h),B_{1}(b)\o B_{2}(g),B_{1}(c)\o B_{2}(f)]_{C\o D}\\
	&=&  [B(a\o h),B(b\o g),B(c\o f)]_{C\o D},
\end{eqnarray*}
and
\begin{eqnarray*}
	&&[B(a\o h)_{1},B(B(a\o h)_{4}),B(B(a\o h)_{2})]_{C\o D}\o B(a\o h)_{3}\\
	&=& [B_{1}(a)_{1}\o B_{2}(h)_{1},B(B_{1}(a)_{4}\o B_{2}(h)_{4}),B(B_{1}(a)_{2}\o B_{2}(h)_{2})]_{C\o D}\o B_{1}(a)_{3}\o B_{2}(h)_{3}\\
	&=& [B_{1}(a)_{1}\o B_{2}(h)_{1},B_{1}(B_{1}(a)_{4})\o B_{2}(B_{2}(h)_{4}),B_{1}(B_{1}(a)_{2})\o B_{2}(B_{2}(h)_{2})]_{C\o D}\o B_{1}(a)_{3}\o B_{2}(h)_{3}\\
	&=& [B_{1}(a)_{1},B_{1}(B_{1}(a)_{4}),B_{1}(B_{1}(a)_{2})]_{C}\o [B_{2}(h)_{1},B_{2}(B_{2}(h)_{4}),B_{2}(B_{2}(h)_{2})]_{D}\o B_{1}(a)_{3}\o B_{2}(h)_{3}\\
    &=&(\textnormal{id}\otimes \tau\otimes \textnormal{id})([B_{1}(a)_{1},B_{1}(B_{1}(a)_{4}),B_{1}(B_{1}(a)_{2})]_{C}\o B_{1}(a)_{3}\o [B_{2}(h)_{1},B_{2}(B_{2}(h)_{4}),B_{2}(B_{2}(h)_{2})]_{D}\\
    && \o B_{2}(h)_{3})\\
	&=&(\textnormal{id}\otimes \tau\otimes \textnormal{id})( B_{1}(a_{1})\o B_{1}(a_{2})\o B_{2}(h_{1})\o B_{2}(h_{2}))\\
	&=&  B(a_{1}\o h_{1})\o B(a_{2}\o h_{2})\\
	&=&  B((a\o h)_{1})\o B((a\o h)_{2}).
\end{eqnarray*}
\end{enumerate}
\end{exam}

\begin{prop}\mlabel{prop:hopfheap2}
Let $B$ be a Rota-Baxter operator on the Hopf heap $(C, \D, \chi)$. Then $B^{(\psi)}=\psi\circ B\circ\psi^{-1}$ is also a Rota-Baxter operator on the Hopf heap $(C, \D, \chi)$, if there exists a Hopf heap automorphism $\psi:C\rightarrow C$.

\begin{proof} It is easy to see that $\psi^{-1}$ is a Hopf
heap homorphism. Moreover, for all $a,b,c\in C$, we have
\begin{eqnarray*}
	B^{(\psi)}([a,b,c])
	&=& \psi(B(\psi^{-1}([a,b,c])))\\
	%&=& \psi( B([\psi^{-1}(a),\psi^{-1}(b),\psi^{-1}(c)]))\\
	%&=& \psi([ B(\psi^{-1}(a)), B(\psi^{-1}(b)), B(\psi^{-1}(c))])\\
	&=& [ \psi(B(\psi^{-1}(a))), \psi(B(\psi^{-1}(b))), \psi(B(\psi^{-1}(c)))]\\
	&=&[B^{(\psi)}(a),	B^{(\psi)}(b),	B^{(\psi)}(c)],
\end{eqnarray*}
and
\begin{eqnarray*}	&&[B^{(\psi)}(a)_{1},B^{(\psi)}(B^{(\psi)}(a)_{4}),B^{(\psi)}(B^{(\psi)}(a)_{2})]\o B^{(\psi)}(a)_{3} \\
	&=&[\psi(B(\psi^{-1}(a)))_{1},\psi(B(\psi^{-1}(\psi(B(\psi^{-1}(a)))_{4}))),\psi(B(\psi^{-1}(\psi(B(\psi^{-1}(a)))_{2})))]\o \psi(B(\psi^{-1}(a)))_{3} \\
	&=&[\psi(B(\psi^{-1}(a))_{1}),\psi(B(\psi^{-1}(\psi(B(\psi^{-1}(a))_{4})))),\psi(B(\psi^{-1}(\psi(B(\psi^{-1}(a))_{2}))))]\o \psi(B(\psi^{-1}(a))_{3}) \\
	&=&(\psi\o\psi)(\underbrace{[B(\psi^{-1}(a))_{1},B(B(\psi^{-1}(a))_{4}),B(B(\psi^{-1}(a))_{2})]\o B(\psi^{-1}(a))_{3}})\\
	&=& (\psi\o\psi)(B(\psi^{-1}(a)_{1})\o B(\psi^{-1}(a)_{2}))\\
	&=& (\psi\o\psi)(B(\psi^{-1}(a_{1}))\o B(\psi^{-1}(a_{2})))\\
	&=& \psi(B(\psi^{-1}(a_{1})))\o \psi(B(\psi^{-1}(a_{2})))\\
		&=& B^{(\psi)}(a_{1})\o B^{(\psi)}(a_{2}).
\end{eqnarray*}
Hence $B^{(\psi)}$ is a Rota-Baxter operator on the Hopf heap $(C, \D, \chi)$.
\end{proof}

\end{prop}

\begin{prop}\mlabel{prop:hopfheap3}
Let $B$ be a Rota-Baxter operator on the commutative Hopf heap $(C, \D, \chi)$. Then
$\tau^{y}_{x}\circ B$ is also a Rota-Baxter operator on the Hopf heap $(C, \D, \chi)$ if $\v\circ B=\v$ and $B(x)\in G(C)$, for all $x\in G(C)$.

Here $\tau^{y}_{x}:C\rightarrow C,c\mapsto [c,x,y]$ is a right $(x,y)$-translation, for any given $x,y\in G(C)$.

In particular, if $B$ is a coalgebra homomorphism, we know that $\tau^{y}_{x}\circ B$ is a Rota-Baxter operator on the Hopf heap $(C, \D, \chi)$.

\begin{proof} According to Lemma {\mref{lem:Hppro}}, we have
\begin{eqnarray*}
\tau^{y}_{x}(B([a,b,c]))
&=&[[B(a),B(b),B(c)],x,y]\\
&=&[[B(a),B(b),B(c)],[x,x,x],[y,y,y]]\\
&=&[[B(a),x,y],[B(b),x,y],[B(c),x,y]]\\
&=&[\tau^{y}_{x}(B(a)),\tau^{y}_{x}(B(b)),\tau^{y}_{x}(B(c))],
\end{eqnarray*}
for all $a,b,c\in C$, and obeying by the condition ``$\v\circ B=\v$ and $B(x)\in G(C)"$, we get that
\begin{eqnarray*}	&&[\tau^{y}_{x}(B(a))_{1},\tau^{y}_{x}(B(\tau^{y}_{x}(B(a))_{4})),\tau^{y}_{x}(B(\tau^{y}_{x}(B(a))_{2}))]\o \tau^{y}_{x}(B(a))_{3}\\	&=&[[B(a),x,y]_{1},\tau^{y}_{x}(B([B(a),x,y]_{4})),\tau^{y}_{x}(B([B(a),x,y]_{2}))]\o [B(a),x,y]_{3}\\	&=&[[B(a)_{1},x,y],\tau^{y}_{x}(B([B(a)_{4},x,y])),\tau^{y}_{x}(B([B(a)_{2},x,y]))]\o [B(a)_{3},x,y]\\	&=&[[B(a)_{1},x,y],\tau^{y}_{x}([B(B(a)_{4}),B(x),B(y)]),\tau^{y}_{x}([B(B(a)_{2}),B(x),B(y)])]\o [B(a)_{3},x,y]\\	&=&[[B(a)_{1},x,y],[[B(B(a)_{4}),B(x),B(y)],x,y],[[B(B(a)_{2}),B(x),B(y)],x,y]]\o [B(a)_{3},x,y]\\
	&=&[[B(a)_{1},[B(B(a)_{4}),B(x),B(y)],[B(B(a)_{2}),B(x),B(y)]],x,y]\o [B(a)_{3},x,y]\\	&=&[[[B(a)_{1},x,x],[B(B(a)_{4}),B(x),B(y)],[B(B(a)_{2}),B(x),B(y)]],x,y]\o [B(a)_{3},x,y]\\	&=&[[[B(a)_{1},B(B(a)_{4}),B(B(a)_{2})],[x,B(x),B(x)],[x,B(y),B(y)]],x,y]\o [B(a)_{3},x,y]\\
		&=&[[[B(a)_{1},B(B(a)_{4}),B(B(a)_{2})],\v(B(x))x,\v(B(y))x],x,y]\o [B(a)_{3},x,y]\\
	&=&[[B(a)_{1},B(B(a)_{4}),B(B(a)_{2})],x,y]\o [B(a)_{3},x,y]\\
	&=&[B(a_{1}),x,y]\o [B(a_{2}),x,y]\\
	&=& \tau^{y}_{x}(B(a_{1}))\o \tau^{y}_{x}(B(a_{2})).
\end{eqnarray*}

Thus $\tau^{y}_{x}\circ B$ is a Rota-Baxter operator on the Hopf heap $(C, \D, \chi)$.
\end{proof}

\end{prop}

\begin{defn}\mlabel{def:rbohopfalg}\mcite{Zheng}
Let $(H,\mu,\eta, \D, \varepsilon, S)$ be a commutative Hopf algebra. An algebra map $B:H\rightarrow H$ is called a Rota-Baxter co-operator on $H$ if for all $x\in H,$
\begin{eqnarray}
B(x_{1})\o B(x_{2})=B(x)_{1}B(B(x)_{2}S(B(x)_{4}))\o B(x)_{3}.
\end{eqnarray}
By a Rota-Baxter commutative Hopf algebra we mean a pair $(H,B)$ of a commutative Hopf algebra and a Rota-Baxter co-operator $B$ on $H$.
\end{defn}

In the following, we build a connection between a Rota-Baxter operator on a Hopf heap and a Rota-Baxter co-operator on a Hopf algebra.

\begin{prop}\mlabel{prop:rbohopfheaphopfalg}
Let $H$ be a commutative Hopf algebra with antipode $S$. If $B$ is a Rota-Baxter co-operator on the Hopf algebra $H$ satisfying $S\circ B=B\circ S$, then $B$ is also a Rota-Baxter operator on the Hopf heap $Hp(H,[-,-,-])$.

Here the Hopf heap map $[-,-,-]$ as in Example \mref{exam:hopfheaphopfalg} is given by
$$[a,b,c]=aS(b)c,$$
for $a,b,c\in H$.

\begin{proof} We have to prove that $B$ is a Rota-Baxter operator on the Hopf heap $Hp(H,[-,-,-])$.

As a matter of facts, for all $a,b,c\in H$, we have
\begin{eqnarray*}
	B([a,b,c])&=&B(aS(b)c)=B(a)B(S(b))B(c)=B(a)S(B(b))B(c)\\
&=&[B(a),B(b),B(c)],\\
B(a_{1})\o B(a_{2})
	&=&B(a)_{1}B(B(a)_{2}S(B(a)_{4}))\o B(a)_{3}\\
	&=&B(a)_{1}B(S(B(a)_{4}))B(B(a)_{2})\o B(a)_{3}\\
	&=&B(a)_{1}S(B(B(a)_{4}))B(B(a)_{2})\o B(a)_{3}\\
	&=&[B(a)_{1},B(B(a)_{4}),B(B(a)_{2})]\o B(a)_{3}.
\end{eqnarray*}

Hence $B$ is a Rota-Baxter operator on the Hopf heap $Hp(H,[-,-,-])$.
\end{proof}

\end{prop}

\begin{prop}\mlabel{prop:rbohopfalghopfheap}
Let $B$ be a Rota-Baxter operator on a commutative Hopf heap $(C, \D, \chi)$. If $B(x)=x$, for any $x\in G(C)$, then $B$ is a Rota-Baxter co-operator on the Hopf algebra $H_x(C)$.

Here the multiplication and antipode of the Hopf algebra $H_x(C)$ as in Example \mref{exam:hopfheaphopfalg} are given by
$$ab=[a,x,b], S(a)=[x,a,x].$$

\begin{proof}  Now we prove that $B$ is a Rota-Baxter co-operator on the Hopf algebra $H_{x}(C)$.

In fact, since for all $a,b\in C$,
$$ab=[a,x,b]=[b,x,a]=ba,$$
$$B(ab)=B([a,x,b])=[B(a),B(x), B(b)]=[B(a),x, B(b)]=B(a)B(b),$$
$H_{x}(C)$ is commutative, and $B$ an algebra map on $H_{x}(C)$.

Moreover, we have
\begin{eqnarray*}
	B(a_{1})\o B(a_{2})	
	&=&[B(a)_{1},B(B(a)_{4}),B(B(a)_{2})]\o B(a)_{3}\\
	&=&[B(a)_{1},B(B(a)_{4}),[x,x,B(B(a)_{2})]]\o B(a)_{3}\\
	&=&[[B(a)_{1},B(B(a)_{4}),x],x,B(B(a)_{2})]\o B(a)_{3}\\
	&=&[[[B(a)_{1},x,x],B(B(a)_{4}),x],x,B(B(a)_{2})]\o B(a)_{3}\\
	&=&[[B(a)_{1},x,[x,B(B(a)_{4}),x]],x,B(B(a)_{2})]\o B(a)_{3}\\
	&=&[B(a)_{1},x,[\underbrace{[x,B(B(a)_{4}),x],x,B(B(a)_{2})]}]\o B(a)_{3}\\
	&=&[B(a)_{1},x,[B(B(a)_{2}),x,[x,B(B(a)_{4}),x]]]\o B(a)_{3}\\
	&=&[[B(a)_{1},x,B(B(a)_{2})],x,[\underbrace{x,B(B(a)_{4}),x}]]\o B(a)_{3}\\
	&=&[[B(a)_{1},x,B(B(a)_{2})],x,[B(x),B(B(a)_{4}),B(x)]]\o B(a)_{3}\\
	&=&[[B(a)_{1},x,B(B(a)_{2})],x,B([x,B(a)_{4},x])]\o B(a)_{3}\\
	&=&B(a)_{1}B(B(a)_{2})B([x,B(a)_{4},x])\o B(a)_{3}\\
	&=&B(a)_{1}B(B(a)_{2})B(S(B(a)_{4}))\o B(a)_{3}\\
	&=&B(a)_{1}B(B(a)_{2}S(B(a)_{4}))\o B(a)_{3}.
\end{eqnarray*}

Hence $B$ is a Rota-Baxter co-operator on the Hopf algebra $H_{x}(C)$.
\end{proof}

\end{prop}

\begin{remark}\mlabel{remark:rbohopfalghopfheap1}
\begin{enumerate}
\item Let $(H, B)$ be a Rota-Baxter cocommutative Hopf algebra with antipode $S$, and $S\circ B=B\circ S$. Then, we easily prove that $B^{\circ}$ is a Rota-Baxter co-operator on the commutative Hopf algebra $H^{\circ}$ with antipode $S^{\circ}$, such that $S^{\circ}\circ B^{\circ}=B^{\circ}\circ S^{\circ}$.

By Proposition \mref {prop:rbohopfheaphopfalg}, $B^{\circ}$ is also a Rota-Baxter operator on the Hopf heap $Hp(H^{\circ},[-,-,-]_{\circ})$, where the Hopf heap map is given by $[\alpha,\beta,\gamma]_{\circ}=\alpha S^{\circ}(\beta)\gamma$, for $\alpha, \beta, \gamma\in H^{\circ}$.

 \item Let $B$ be a Rota-Baxter operator on a commutative Hopf heap $(C, \D, \chi)$, and $B(x)=x$ for any $x\in G(C)$. Then, according to Proposition \mref {prop:rbohopfalghopfheap}, $B$ is a Rota-Baxter co-operator on the Hopf algebra $H_x(C)$. So, $B^{\circ}$ is a Rota-Baxter operator on the Hopf algebra $H_x(C)^{\circ}$.
\end{enumerate}
\end{remark}

\subsection{Descendent Hopf heaps}

In this subsection, we construct a Hopf heap structure by a Rota-Baxter operator on commutative Hopf heap.

\begin{prop}\mlabel{prop:rbohopfheapcoalg}
	Let $B$ be a Rota-Baxter operator on commutative Hopf heap $(C, \D, \chi)$. Define
	$$\D^{'}(a)=a_{1^{'}}\o a_{2^{'}}=[a_{1}, B(a_{4}), B(a_{2})]\o a_{3},$$
	for all $a\in C$. Then $(C, \D^{'}, \v)$ is a coalgebra if $B$ is a surjective map such that  $\varepsilon\circ B=\varepsilon$.
\end{prop}

\begin{proof} Because $B$ is a Rota-Baxter operator on commutative Hopf heap, we get $\D^{'}B=(B\o B)\D$. Since $B$ is a surjective map, there is an element $\upsilon\in C$ with $a=B(\upsilon)$, for any $a\in C$.

Next, we verify that
	\begin{eqnarray*}
		a_{1^{'}}\varepsilon(a_{2^{'}})
		&=& [a_{1}, B(a_{3}), B(a_{2})]\\
        &=&[B(\upsilon)_{1}, B(B(\upsilon)_{3}), B(B(\upsilon)_{2})]\\
		&=& (\textnormal{id}\o \v)([B(\upsilon)_{1}, B(B(\upsilon)_{4}),B(B(\upsilon)_{2})]\o B(\upsilon)_{3})\\
		&=& (\textnormal{id}\o \v)(B(\upsilon_{1})\o B(\upsilon_{2}))\\
        &=&B(\upsilon)\\
        &=&a,\\
			\varepsilon(a_{1^{'}})a_{2^{'}}
		&=& \varepsilon([a_{1}, B(a_{4}), B(a_{2})])a_{3}\\
        &=&\varepsilon(a_{1})\varepsilon(B(a_{4}))\varepsilon(B(a_{2}))a_{3}\\
		&=&\varepsilon(a_{1})\varepsilon(a_{4})\varepsilon(a_{2})a_{3}\\
        &=&a,\\	
		(\D^{'}\o \textnormal{id})\D^{'}(a)
		&=& (\D^{'}\o \textnormal{id})\D^{'}(B(\upsilon))\\
        &=&(\D^{'}\o \textnormal{id})(B(\upsilon_{1})\o B(\upsilon_{2}))\\
        &=& \D^{'}B(\upsilon_{1})\o B(\upsilon_{2})\\		
		&=& B(\upsilon_{11})\o B(\upsilon_{12})\o B(\upsilon_{2})\\
        &=& B(\upsilon_{1})\o B(\upsilon_{21})\o B(\upsilon_{22})\\
        &=& (\textnormal{id}\o (B\o B)\circ \D)(B(\upsilon_{1})\o \upsilon_{2})\\	
		&=& (\textnormal{id}\o \D^{'}\circ B)(B(\upsilon_{1})\o \upsilon_{2})\\
        &=& (\textnormal{id}\o \D^{'})(B(\upsilon_{1})\o B(\upsilon_{2}))\\
        &=& (\textnormal{id}\o \D^{'})(B\o B)\D(\upsilon)\\			
        &=& (\textnormal{id}\o \D^{'})\D^{'}B(\upsilon)\\
        &=&(\textnormal{id}\o \D^{'})\D^{'}(a).
	\end{eqnarray*}
	
Hence $(C, \D^{'}, \v)$ is a coalgebra.
\end{proof}

\begin{theorem}\mlabel{descendenthopfalg}

 $C_B=(C, \D^{'}, \chi)$ given in Proposition \mref{prop:rbohopfheapcoalg} is a Hopf heap.

\end{theorem}

\begin{proof} Because $B$ is a surjective map, there are $f, g,h\in C$ with $a=B(f), b=B(g), c=B(h)$, for any $a,b,c\in C$, respectively.

First, we prove that
	\begin{eqnarray*}
		[a_{1^{'}},a_{2^{'}},b] &=& [B(f)_{1^{'}},B(f)_{2^{'}},b]= [B(f_{1}),B(f_{2}),b]= B([f_{1}, f_{2}, g])\\
		&=& \varepsilon(f)B(g)= \varepsilon(B(f))b= \varepsilon(a)b,
	\end{eqnarray*}
	and
	\begin{eqnarray*}
		[b,a_{1^{'}},a_{2^{'}}] &=& [b,B(f)_{1^{'}},B(f)_{2^{'}}]= [b,B(f_{1}),B(f_{2})]= B([g, f_{1}, f_{2}])\\
		&=& \varepsilon(f)B(g)= \varepsilon(B(f))b= \varepsilon(a)b.
	\end{eqnarray*}
	In addition, we have to check that $\chi$ is a coalgebra homomorphism with respect to $\D^{'}$.
\begin{eqnarray*}
	\D^{'}[a,b,c]
	&=&\D^{'}B[f,g,h]\\
	&=&(B\o B)\D[f,g,h]\\
	&=&B([f_{1},g_{2},h_{1}])\o B([f_{2},g_{1},h_{2}])\\
	&=&[B(f_{1}),B(g_{2}),B(h_{1})]\o [B(f_{2}),B(g_{1}),B(h_{2})]\\
	&=&(\chi\o\chi)(B(f_{1})\o B(g_{2})\o B(h_{1})\o B(f_{2})\o B(g_{1})\o B(h_{2}))\\
	&=&(\chi\o\chi)(\textnormal{id}\o \textnormal{id}\o T\o \textnormal{id}\o \textnormal{id})(\textnormal{id}\o T\o T\o \textnormal{id})(\textnormal{id}\o \textnormal{id}\o T\o \textnormal{id}\o \textnormal{id})\\
	&& (B(f_{1})\o  B(f_{2})\o  B(g_{1})\o B(g_{2})\o  B(h_{1})\o B(h_{2}))\\
	&=&(\chi\o\chi)(\textnormal{id}\o \textnormal{id}\o T\o \textnormal{id}\o \textnormal{id})(\textnormal{id}\o T\o T\o \textnormal{id})(\textnormal{id}\o \textnormal{id}\o T\o \textnormal{id}\o \textnormal{id})\\
	&&(B\o B\o B\o B\o B\o B)(\D\o \D\o \D)(f\o g\o h)\\
	&=&(\chi\o\chi)(\textnormal{id}\o \textnormal{id}\o T\o \textnormal{id}\o \textnormal{id})(\textnormal{id}\o T\o T\o \textnormal{id})(\textnormal{id}\o \textnormal{id}\o T\o \textnormal{id}\o \textnormal{id})\\
    &&((B\o B)\circ\D\o (B\o B)\circ\D\o (B\o B)\circ\D)(f\o g\o  h)\\
    &=&(\chi\o\chi)(\textnormal{id}\o \textnormal{id}\o T\o \textnormal{id}\o \textnormal{id})(\textnormal{id}\o T\o T\o \textnormal{id})(\textnormal{id}\o \textnormal{id}\o T\o \textnormal{id}\o \textnormal{id})\\
    &&(\D^{'}\circ B\o \D^{'}\circ B\o \D^{'}\circ B)(f\o g\o h)\\
	&=&(\chi\o\chi)(\textnormal{id}\o \textnormal{id}\o T\o \textnormal{id}\o \textnormal{id})(\textnormal{id}\o T\o T\o \textnormal{id})(\textnormal{id}\o \textnormal{id}\o T\o \textnormal{id}\o \textnormal{id})\\
	&&(\D^{'}\o \D^{'}\o \D^{'})(B\o B\o B)(f\o g\o h)\\
	&=&(\chi\o\chi)(\textnormal{id}\o \textnormal{id}\o T\o \textnormal{id}\o \textnormal{id})(\textnormal{id}\o T\o T\o \textnormal{id})(\D^{'}\o \D^{'co}\o \D^{'})(a\o b\o c)\\
&=&[a_{1^{'}},b_{2^{'}},c_{1^{'}}]\o [a_{2^{'}},b_{1^{'}},c_{2^{'}}].
\end{eqnarray*}
Hence, $(C, \D^{'}, \chi)$ is a Hopf heap.
\end{proof}

\begin{defn}\mlabel{def:rbohopfalg}
The Hopf heap $C_{B}$ given in Theorem \mref{descendenthopfalg} is called the descendent Hopf heap.
\end{defn}
%\begin{exam}
%	Let $C$ be a vector space with basis $\{u,\theta\}$. Define
%	$$
%	\D(u)=u\o \theta+ \theta \o u, ~~\D(\theta)=\theta\o \theta-u\o u,
%	$$
%	$$\v(u)=0, ~~\v(\theta)=1.
%	$$
%	$$
%	[u,u,u]=-u, [\theta,\theta,\theta]=\theta, ~\text{else is zero}.
%	$$
	%
%	By Example \mref{exam:rbohopfheap5}(1), $B(u)=-u, B(\theta)=\theta$ is a 	Rota-Bater operator on the $(C,\D,\chi)$ such that $B$ is an invertible map and $\varepsilon\circ B=\varepsilon$, then define
%	$$
%	\D^{'}(u)=[\theta, \theta, \theta ]\o u+ [-u, -u,-u]\o \theta=\theta\o u+u\o \theta
%	$$
%	$$
%	\D^{'}(\theta)=[\theta, \theta, \theta ]\o \theta+ [u, -u, -u]\o u=\theta\o\theta-u\o u
%	$$
%\end{exam}

\begin{prop}
$B$ is a homomorphism from a Hopf heap $C$ to the descendent Hopf heap $C_{B}$, and also a Rota-Baxter operator on $C_{B}$.
\end{prop}

\begin{proof}
First we note that $\D^{'}B=(B\o B)\D$ and $\varepsilon\circ B=\varepsilon$. Then $B$ is a coalgebra map from  a Hopf heap $C$ to the descendent Hopf heap $C_{B}$, and $B([a,b,c])=[B(a),B(b),B(c)]$, for all $a,b,c\in C$. Therefore, $B$ is a Hopf heap homomorphism from $C$ to $C_{B}$.

Moreover, we can prove that $B$ is a Rota-Baxter operator on the descendent Hopf heap $C_{B}$: for any $a\in C$, we have
\begin{eqnarray*}
B(a_{1^{'}})\o B(a_{2^{'}})
	&=&B[a_{1},B(a_{4}),B(a_{2})]\o B(a_{3})\\
	&=&[B(a_{1}),B(B(a_{4})),B(B(a_{2}))]\o B(a_{3})\\
	&=&[B(a)_{1^{'}},B(B(a)_{4^{'}}),B(B(a)_{2^{'}})]\o B(a)_{3^{'}}.
\end{eqnarray*}
\end{proof}

\begin{prop}\mlabel{prop:isoHopfheap}
Let $B$ be a Rota-Baxter operator on commutative Hopf heap $(C, \Delta, \chi)$. If $B$ is  a surjective map such that $\varepsilon\circ B=\varepsilon$, then $(C_{B}, \Delta^{'}_{B}, \chi)$ $\cong$ $(C_{B^{(\psi)}}, \Delta^{'}_{B^{(\psi)}}, \chi)$ as Hopf heaps, where $\psi:C\rightarrow C$ is a Hopf heap automorphism, and $B^{(\psi)}$ is defined by Proposition \mref{prop:hopfheap2}.
\end{prop}

\begin{proof}
By Proposition \mref{prop:hopfheap2}, we know $B^{(\psi)}$ is a Rota-Baxter operator on $C$.

It is easy to see that $\varepsilon\circ B^{(\psi)}=\varepsilon\circ\psi\circ B \circ \psi^{-1}=\varepsilon,$
and $B^{(\psi)}$ is surjective, so,  by Proposition \mref{prop:rbohopfheapcoalg} and Theorem \mref{descendenthopfalg}, $C_{B^{(\psi)}}$ is a Hopf heap.

In addition, for any $a\in C$,
\begin{eqnarray*}
	\Delta^{'}_{B^{(\psi)}}(\psi(a))
	&=&[\psi(a)_{1},B^{(\psi)}(\psi(a)_{4}),B^{(\psi)}(\psi(a)_{2})]\o \psi(a)_{3}\\
	&=&[\psi(a_{1}),B^{(\psi)}(\psi(a_{4})),B^{(\psi)}(\psi(a_{2}))]\o \psi(a_{3})\\
	&=&[\psi(a_{1}),\psi( B(a_{4})),\psi( B(a_{2}))]\o \psi(a_{3})\\
	&=&(\psi\o\psi)([a_{1}, B(a_{4}),B(a_{2})]\o a_{3})\\
	&=&(\psi\o\psi)(\Delta^{'}_{B}(a)).
\end{eqnarray*}

Hence, $\psi$ is an isomorphism from the Hopf heap $C_{B}$ to the Hopf heap $C_{B^{(\psi)}}$.
\end{proof}

In what follows, we prove that a Hopf co-brace (see \mcite{AGV,Zheng1}) can be induced by a Rota-Baxter operator on commutative Hopf heap.

\begin{prop}
Let $B$ be a Rota-Baxter operator on commutative Hopf heap $(C, \Delta, \chi)$, $x\in G(C)$. If $B$ is a surjective map such that $\varepsilon\circ B=\varepsilon$ and $B(x)=x$, then the following conclusions hold:
\begin{enumerate}
\item $(H_{x}(C),\Delta,\Delta^{'}_{B})$ is a Hopf co-brace.

 \item There is an isomorphism: $(H_{x}(C),\Delta,\Delta^{'}_{B})$ $\cong$ $(H_{x}(C),\Delta,\Delta^{'}_{B^{(\psi)}})$ as Hopf co-braces if there exists a Hopf heap automorphism $\psi:C\rightarrow C$ such that $\psi(x)=x$.

     \end{enumerate}
\end{prop}

\begin{proof} (a) By Theorem \mref{descendenthopfalg}, $(C,\Delta^{'}_{B},\chi)$ is a Hopf heap, and by Example \mref{exam:hopfheaphopfalg}, $(H_{x}(C), \Delta^{'}_{B})$ a Hopf algebra.

Moreover, for any $h\in C$, we have
\begin{eqnarray*}
	h_{11^{'}}S(h_{2})h_{31^{'}}\o h_{12^{'}}\o h_{32^{'}}
	&=&[[h_{11^{'}},x,S(h_{2})],x,h_{31^{'}}]\o h_{12^{'}}\o h_{32^{'}}\\
	&=&[[h_{11^{'}},x,[x,h_{2},x]],x,h_{31^{'}}]\o h_{12^{'}}\o h_{32^{'}}\\
	&=&[h_{11^{'}},h_{2},h_{31_{'}}]\o h_{12^{'}}\o h_{32^{'}}\\
	&=&[[h_{11}, B(h_{14}),B(h_{12})],h_{2},[h_{31}, B(h_{34}),B(h_{32})]]\o h_{13}\o h_{33}\\
    &=&[[h_{1}, B(h_{4}),B(h_{2})],h_{5},[h_{6}, B(h_{9}),B(h_{7})]]\o h_{3}\o h_{8}\\
    &=&[[[h_{1}, B(h_{4}),B(h_{2})],h_{5},h_{6}], B(h_{9}),B(h_{7})]\o h_{3}\o h_{8}\\
    &=&[[h_{1}, B(h_{4}),B(h_{2})], B(h_{7}),B(h_{5})]\o h_{3}\o h_{6}\\
    &=&[h_{1}, B(h_{4}),[B(h_{2}), B(h_{7}),B(h_{5})]]\o h_{3}\o h_{6}\\
    &=&[h_{1}, B(h_{4}),[B(h_{5}), B(h_{7}),B(h_{2})]]\o h_{3}\o h_{6}\\
    &=&[[h_{1}, B(h_{4}),B(h_{5})], B(h_{7}),B(h_{2})]\o h_{3}\o h_{6}\\
    &=&[[h_{1}, B(h_{41}),B(h_{42})], B(h_{6}),B(h_{2})]\o h_{3}\o h_{5}\\
    &=&[[h_{1}, B(h_{4})_{1'},B(h_{4})_{2'}], B(h_{6}),B(h_{2})]\o h_{3}\o h_{5}\\
    &=&[h_{1}, B(h_{5}),B(h_{2})]\o h_{3}\o h_{4}\\
	&=&h_{1^{'}}\o h_{2^{'}1}\o h_{2^{'}2}.
\end{eqnarray*}

Hence, $(H_{x}(C),\Delta,\Delta^{'}_{B})$ is a Hopf co-brace.

(b) By (1), $(H_{x}(C),\Delta,\Delta^{'}_{B^{(\psi)}})$ is also a Hopf co-brace. Again by Proposition \mref{prop:isoHopfheap}, $\psi$ is a coalgebra homomorphism from $(H_{x}(C),\Delta^{'}_{B})$ to $(H_{x}(C),\Delta^{'}_{B^{(\psi)}})$.

Moreover, for any $a,b\in C$,
$$\psi(ab)=\psi([a,x,b])=[\psi(a),\psi(x),\psi(b)]
=[\psi(a),x,\psi(b)]=\psi(a)\psi(b),$$
that is, $\psi$ is an algebra homomorphism.

Hence $(H_{x}(C),\Delta,\Delta^{'}_{B})$ $\cong$ $(H_{x}(C),\Delta,\Delta^{'}_{B^{(\psi)}})$ as Hopf co-braces.
\end{proof}

\vspace{2mm}

\section{Rota-Baxter Hopf heap modules and their structure theorems}\mlabel{sec:rbohpm}

\vspace{2mm}

In this section, we introduce the conception of Hopf heap modules, and give its structure theorem.

\begin{defn}\mlabel{def:hopfheapmodule}
	Let $(H,\chi)$ be a Hopf heap, and $(M,\r)$ a right $H$-comodule (denote $\r(m)=m_{(0)}\o m_{(1)}$). We call $M$ a right $H$-Hopf heap module if there is an action $\lhd: M\o H^{co}\o H\rightarrow M$ by $m\lhd (c\o d)$, such that for any $a,b,c,d\in H, m\in M$,
\begin{eqnarray}
	&&m\lhd (a\o [b, c,d])=(m\lhd (a\o b))\lhd (c\o d),\\
	&&m\lhd (c_1\o c_2)=\v(c)m,\\
	&&\r(m\lhd (c\o d))=m_{(0)}\lhd (c_2\o d_1)\o [m_{(1)},c_1,d_2,].
	\end{eqnarray}
	In the following, we denote the right $H$-Hopf heap module $M$ by $(M,\lhd,\r)$ or $M$.
\end{defn}

	Let $(M,\lhd_1,\r_1)$ and $(N,\lhd_2,\r_2)$ be two right $H$-Hopf heap modules and $f: M\rightarrow N$ a linear map. We call $f$ a right $H$-Hopf heap module map if it is a right $H$-comodule map and we have
	$$f(m\lhd_1(a\o b))=f(m)\lhd_2(a\o b)
	$$ for any $a,b\in H, m\in M$.

\begin{remark}\mlabel{rem:rthestructuretheorem}
	(a) Let $(H, \chi)$ be a Hopf heap, and $M$ a vector space. Similar to the case of the left Hopf heap module, we have $M\o H$ is a right $H$-Hopf heap module by $(m\o h)\lhd' (g\o l)=m\o [h,g,l]$ and $\rho'(m\o h)=m\o h_{1}\o h_{2}$, for all $h,g,l\o H, m\in M$.
	
	(b) Let $M$ be a right $H$-Hopf heap module, $x\in G(H)$. Then $M$ is a right $H_x(H)$-Hopf module by defining the action as follows:
	$$m\cdot h= m\lhd (x\o h), \ \ h\in H, m\in M.$$
	
Conversely, if $M$ is a right $H$-Hopf module, then $M$ is a right $Hp(H,[-,-,-])$-Hopf heap module by defining the action as follows:
	$$ m\lhd (a\o b)=m\cdot(S(a)b),\ \ a,b\in H, m\in M.$$
	
	(c)  Let $(H, \chi)$ be a Hopf heap, $x\in G(H)$. Then there exists an isomorphism of right $H$-Hopf heap modules: $M\cong  M^{coH}_x\o H,$ by two maps as follows:
	$$\alpha: M_{x}^{coH}\o H\rightarrow M,~ ~ m\o h\mapsto m\lhd(x\o h),$$
		$$\beta: M\rightarrow M_{x}^{coH}\o H,~ ~ m\mapsto P(m_{(0)})\o m_{(1)},$$ where $P(m)=m_{(0)}\lhd (m_{(1)}\o x)$, for $m\in M$.

In fact, we have only to prove that $\mathrm{Im} (P(m))\subseteq M_{x}^{coH}$:
\begin{eqnarray*}
\r(P(m))&=&\r(m_{(0)}\lhd (m_{(1)}\o x))\\
&=& m_{(0)(0)}\lhd (m_{(1)2}\o x)\o [m_{(0)(1)},m_{(1)1},x]\\
&=& m_{(0)}\lhd (m_{(1)22}\o x)\o [m_{(1)1},m_{(1)21},x]\\
&=& m_{(0)}\lhd (m_{(1)}\o x)\o x\\
&=& P(m)\o x.
\end{eqnarray*}

\end{remark}

%\begin{defn}\mcite{BH}
%	\color{red}A Grunspan map for a Hopf heap $C$ is a coalgebra homomorphism $\vartheta: C \rightarrow C$,
%	such that, for all $a, b, c, d, e \in C$,
	%$$ [[a, b, \vartheta(c)], d, e] = [a, [d, c, b], e].$$
%	Every Hopf heap admits the Grunspan map. The Grunspan map is given by
%$$\vartheta : C \rightarrow C, c \mapsto\sum [c_{1}, [e_{1}, c_{3}, c_{2}],e_{2}],$$
%	where $e \in C$ is any element such that $\v(e) = 1$.
%\end{defn}
%\begin{prop}
%	\color{red}If $C$ is a commutative Hopf heap and $M$ is a left $C$-Hopf heap module, we have
	%$$[[a,b,c],d,e]=[a,[b,c,d],e]=[a,b,[c,d,e]],$$
%$$([a,b,c]\o d)\rhd m=(a\o [b,c, d])\rhd m,
%$$
	% $$(c_1\o c_2)\rhd m=(c_2\o c_1)\rhd m,$$
%\end{prop}
%\begin{proof}
%	\color{red}For any $x\in G(C), a,b,c,e\in C$, we have
%	\begin{eqnarray*}
%		\vartheta(c)
%		&=&[c_{1}, [e_{1}, c_{3}, c_{2}],e_{2}]=[c_{1}, [c_{2}, c_{3}, e_{1}],e_{2}]=[c,  e_{1},e_{2}]=\v(e)c=c,
%	\end{eqnarray*}
%	hence,
%	$$ [[a, b, c], d, e]=[[a, b, \vartheta(c)], d, e] = [a, [d, c, b], e]=[a, [b, c, d], e].$$
	
%	\begin{eqnarray*}
%		([a,b,c]\o d)\rhd m
%		&=&([a,b,c]\o d)\rhd((x\o x)\rhd m)=([[a,b,c], d,x]\o x)\rhd m\\
%		&=&([a,[b,c, d],x]\o x)\rhd m=(a\o[b,c, d])\rhd((x\o x)\rhd m)\\
%		&=&(a\o [b,c, d])\rhd m,
%	\end{eqnarray*}
%	\begin{eqnarray*}
%		(c_1\o c_2)\rhd m
%		&=&\v(c)m=\v(c) (x\o x)\rhd m= ([x, c_1,c_2]\o x)\rhd m= ([x, c_1,c_2]\o x)\rhd m\\		
%		&=&m\lhd ([c_2,c_1,x] \o x)\rhd m= (c_2\o c_1)\rhd((x\o x)\rhd m)=(c_2\o c_1)\rhd m.
%	\end{eqnarray*}
	
%\end{proof}

\begin{prop}
	Let  $C$ be a commutative Hopf heap, and $M$  a right $C$-Hopf heap module. Then we have
	$$m\lhd (a\o [b, c,d])=m\lhd ([a,b,c]\o d),$$
	$$m\lhd (c_1\o c_2)=m\lhd (c_2\o c_1).$$
\end{prop}
\begin{proof} For any $a,b,c\in C$ and given $x\in G(C)$, we have
	\begin{eqnarray*}
		m\lhd (a\o [b, c,d])
		&=&(m\lhd(x\o x) )\lhd(a\o [b, c,d])=m\lhd(x\o [x,a,[b, c,d]])\\
		&=&m\lhd(x\o [x,[a,b, c],d])=(m\lhd(x\o x))\lhd ([a,b, c]\o d)\\
		&=&m\lhd ([a,b, c]\o d),
	\end{eqnarray*}
	\begin{eqnarray*}
	m\lhd (c_1\o c_2)
	&=&\v(c)m=\v(c)m\lhd (x\o x)=m\lhd ([x, c_1,c_2]\o x)=m\lhd ([c_2, c_1,x]\o x)\\		
	&=&m\lhd (c_2\o [c_1,x, x])=m\lhd (c_2\o c_1).
\end{eqnarray*}
\end{proof}

Next, we will introduce the definition of a Rota-Baxter operator on a Hopf heap module.

\begin{defn}\mlabel{def:rbcohpm}
	Let $B$ be a Rota-Baxter operator on Hopf heap $H$ (that is, $(H, B)$ is a  Rota-Baxter Hopf heap). A  Rota-Baxter Hopf heap module over $(H,B)$ is a pair $(M, T)$ with a right $H$-Hopf heap module $M$ and a linear map $T: M\rightarrow M$ satisfying
	%$$T(m\lhd(a\o b))=T(m)\lhd (B(a)\o B(b)),$$
	$$T(m_{(0)})\o B(m_{(1)})=T(m)_{(0)}\lhd (B(T(m)_{(1)3})\o B(T(m)_{(1)1}))\o T(m)_{(1)2},$$
	for all $ m\in M$.
\end{defn}

Specially, $(H,B)$ is a  Rota-Baxter Hopf heap module over $(H,B)$ if $(H, B)$ is a  Rota-Baxter Hopf heap.
	
	Let $(M, T_{M})$ and $(N, T_{N})$ be two  $(H, B)$-Hopf heap modules. A homomorphism $f: (M, T_{M}) \rightarrow (N, T_{N})$ of  Rota-Baxter Hopf heap modules over $(H,B)$ is a  homomorphism $f: M\rightarrow N$ of right $H$-Hopf heap modules such that $f\circ T_{M}=T_{N}\circ f$.

%\begin{exam}
%	(a)
	
	%(b)Let $(C,B)$ be a Rota-Baxter Hopf heap and $(M, T)$ an right Rota-Baxter Hopf heap module over $(C,B)$. Then for any $\mu\in K$, $(M, \mu T)$  be an right Rota-Baxter Hopf heap module over $(C,B)$.
	
	%(b) Let $M$ be a right $H$-Hopf heap module with action $\lhd$ and coaction $\rho$, and $V$ a vector space. Then $V\o M$ is a right Hopf module, whose module structure map and comodule structure map are given by $\textrm{id}\o \lhd$ and $\textrm{id}\o \rho$, respectively. So, if $(M, T)$  is a  Rota-Baxter Hopf heap module over $(H,B)$, we easily see that $(V\o M, \textrm{id}\o T)$ is a  Rota-Baxter Hopf heap module over $(H,B)$.
%\end{exam}

\begin{lemma}\mlabel{cmhpff}
	Let $(H,B)$ be a Rota-Baxter Hopf heap, and $M$ a vector space. For any linear map $F:M\o H\rightarrow M$, define a linear map
	
	$$\widehat{T}: M\o H\rightarrow M\o H, \widehat{T}(m\o h)=F(m\o h_{1})\o B(h_{2}),$$
for any $m\in M, h\in H$.

	Then $(M\o H, \widehat{T})$ is a Rota-Baxter Hopf heap module over $(H,B)$.
\end{lemma}	
\begin{proof}
By Remark \mref{rem:rthestructuretheorem} (a), we know that $M\o H$ is a right $H$-Hopf heap module.

Moreover, for all $m\in M$, $h\in H$, we obtain
	\begin{eqnarray*}
				&&\widehat{T}(m\o h)_{(\hat{0})}\hat{\lhd} (B(\widehat{T}(m\o h)_{(\hat{1})3})\o B(\widehat{T}(m\o h)_{(\hat{1})1}))\o \widehat{T}(m\o h)_{(\hat{1})2} \\
				&=&	(F(m\o h_{1})\o B(h_{2}))_{(\hat{0})}\hat{\lhd} (B((F(m\o h_{1})\o B(h_{2}))_{(\hat{1})3})\o B((F(m\o h_{1})\o B(h_{2}))_{(\hat{1})1}))\\
				&&\o (F(m\o h_{1})\o B(h_{2}))_{(\hat{1})2} \\
				&=&	(F(m\o h_{1})\o B(h_{2})_{1})\hat{\lhd} (B(B(h_{2})_{23})\o B(B(h_{2})_{21}))\o B(h_{2})_{22}\\
				&=&	F(m\o h_{1}) \o \underbrace{[B(h_{2})_{1}, B(B(h_{2})_{23}), B(B(h_{2})_{21})]\o B(h_{2})_{22}}\\
				&=&	F(m\o h_{1}) \o B(h_{21})\o B(h_{22})\\
				&=&	F(m\o h_{11}) \o B(h_{12})\o B(h_{2})\\
				&=&	\widehat{T}(m\o h_{1})\o B(h_{2})\\
				&=&	\widehat{T}((m\o h)_{(\hat{0})})\o B((m\o h)_{(\hat{1})}).
			\end{eqnarray*}
	%	Hence  $(M\o H, T')$ and $ (M\o H, \widehat{T})$ are right Rota-Baxter Hopf heap modules over $(H,B)$.
	%\end{proof}
%	We have $\tau(x\o y)=y\o x$. $I_{H}$ and $I_{M}$ are identity maps.
	
%	\begin{eqnarray*}
%		&&	(\underbrace{\hat{\lhd}}\o I_{H})(I_{M}\o I_{H}\o m\o I_{H})(I_{M}\o I_{H}\o \tau\o I_{H})(I_{M}\o I_{H}\o I_{H}\o B\o I_{H})(I_{M}\o I_{H}\o B\o\tau)\\
%		&&(I_{M}\o I_{H}\o I_{H}\o \D)(I_{M}\o I_{H}\o \D)\underbrace{\hat{\rho}\widehat{T}}\\
%		&=&(I_{M}\o \chi\o I_{H})(I_{M}\o I_{H}\o m\o I_{H})(I_{M}\o I_{H}\o \tau\o I_{H})(I_{M}\o I_{H}\o I_{H}\o B\o I_{H})(I_{M}\o I_{H}\o B\o\tau)\\
%		&&(I_{M}\o I_{H}\o I_{H}\o \D)(I_{M}\o I_{H}\o \D)(I_{M}\o \D)(I_{M}\o B)(F\o I_{H})(I_{M}\o \D)\\
%		&=&(I_{M}\o \underbrace{((\chi\o I_{H})(I_{H}\o m\o I_{H})(I_{H}\o \tau\o I_{H})(I_{H}\o I_{H}\o B\o I_{H})(I_{H}\o B\o\tau)(I_{H}\o I_{H}\o \D)(I_{H}\o \D)\D B)})\\
%		&&(F\o I_{H})(I_{M}\o \D)\\
%		&=&(I_{M}\o (B\o B)\D)(F\o I_{H})(I_{M}\o \D)\\
%		&=&(I_{M}\o B\o B)\underbrace{(I_{M}\o \D)(F\o I_{H})}(I_{M}\o \D)\\
%		&=&(I_{M}\o B\o B)(F\o I_{H}\o I_{H})(I_{M}\o I_{H}\o \D)(I_{M}\o \D)\\	
%		&=&(I_{M}\o B\o B)(F\o I_{H}\o I_{H})(I_{M}\o \D \o I_{H})(I_{M}\o \D)\\	
%		&=&((I_{M}\o B)(F\o I_{H})(I_{M}\o \D)\o B)(I_{M}\o \D)\\
%		&=&(\widehat{T}\o B)\hat{\rho}.
%	\end{eqnarray*}
	
	Hence  $(M\o H, \widehat{T})$ is a Rota-Baxter Hopf heap module over $(H,B)$.
\end{proof}

In what follows, we give the structure theorem of Rota-Baxter Hopf heap modules.
\begin{theorem}\mlabel{the:rbhpmst}
	Let $(H,B)$ be a Rota-Baxter commutative Hopf heap, and $(M, T)$ a Rota-Baxter Hopf heap module over $(H, B)$, $x\in G(H)$. Define $$\widehat{T}(m\o h)=T(m\lhd(x\o h_{1}))_{(0)}\lhd(T(m\lhd(x\o h_{1}))_{(1)}\o x) \o B(h_{2}),~~~m\in M, h\in H.$$
	
Then the following conclusions hold.

	(a) $\mathrm{Im}\widehat{T}\subset M_{x}^{coH}\o H$, where $M^{coH}_{x}=\{m\in M, \rho(m)=m\o x\}$.
	
	(b) If $\varepsilon\circ B=\varepsilon$, then there exists an isomorphism of Rota-Baxter Hopf heap modules:
	$$(M, T)\cong( M_{x}^{coH}\o H, \widehat{T}).$$
\end{theorem}
\begin{proof}
(a) For all $m\in M$, $h\in H$, we get
\begin{eqnarray*}
	&&\rho(T(m\lhd(x\o h))_{(0)}\lhd(T(m\lhd(x\o h))_{(1)}\o x))\\
	&=&T(m\lhd(x\o h))_{(0)(0)}\lhd(T(m\lhd(x\o h))_{(1)2}\o x)\o [T(m\lhd(x\o h))_{(0)(1)}, T(m\lhd(x\o h))_{(1)1}, x]\\
	&=&T(m\lhd(x\o h))_{(0)}\lhd(T(m\lhd(x\o h))_{(1)3}\o x)\o [T(m\lhd(x\o h))_{(1)1}, (T(m\lhd(x\o h))_{(1)2},x)]\\
	&=&T(m\lhd(x\o h))_{(0)}\lhd(T(m\lhd(x\o h))_{(1)2}\o x)\o \varepsilon(T(m\lhd(x\o h))_{(1)1})x\\
	&=&T(m\lhd(x\o h))_{(0)}\lhd(T(m\lhd(x\o h))_{(1)}\o x)\o x.
\end{eqnarray*}
 Hence $\mathrm{Im}\widehat{T}\subset M_{x}^{coH}\o H$.

(b) By (a), $\widehat{T}$ is well defined. According to Remark \mref{rem:rthestructuretheorem}, $M_{x}^{coH}\o H$ is a right $H$-Hopf heap module by $(m\o h)\widehat{\lhd} (g\o l)=m\o [h,g,l]$ and $\rho(m\o h)=(m\o h)_{(\hat{0})}\o (m\o h)_{(\hat{1})}=m\o h_{1}\o h_{2}$, for all $h,g,l\in H, m\in M_{x}^{coH}$.

Define two linear maps as follows:
$$\psi: M\o H\rightarrow M,~ ~ m\o h\mapsto m\lhd(x\o h)$$
$$\phi: M\rightarrow M, ~ ~ m\mapsto [x,m,x].$$

 Taking $F=\psi(I_{M}\o \phi)\rho T\psi$, by Lemma \mref{cmhpff}, we have
$$\widehat{T}(m\o h)=T(m\lhd(x\o h_{1}))_{(0)}\lhd(T(m\lhd(x\o h_{1}))_{(1)}\o x) \o B(h_{2}),$$
so, $( M_{x}^{coH}\o H, \widehat{T})$ is a Rota-Baxter Hopf heap module over $(H,B)$.

According to Remark \mref{rem:rthestructuretheorem}, we have a right $H$-Hopf heap module isomorphisms as follow:
$$\alpha: M_{x}^{coH}\o H\rightarrow M,~ ~ m\o h\mapsto m\lhd(x\o h),$$
with the inverse
$$\beta: M\rightarrow M_{x}^{coH}\o H,~ ~ m\mapsto P(m_{(0)})\o m_{(1)}.$$

Moreover, for all $m\in M_{x}^{coH}$, $h\in H$, we have
\begin{eqnarray*}
\rho(m\lhd(x\o h))&=&m_{(0)}\lhd(x\o h_{1})\o [m_{(1)},x,h_{2}]\\
&=&m\lhd(x\o h_{1})\o [x,x,h_{2}]\\
&=&m\lhd(x\o h_{1})\o h_{2}.
\end{eqnarray*}
and
\begin{eqnarray*}
	&&\widehat{T}(m\o h)\\
	&=&T(m\lhd(x\o h_{1}))_{(0)}\lhd(T(m\lhd(x\o h_{1}))_{(1)}\o x) \o B(h_{2})\\
	&=&T((m\lhd(x\o h))_{(0)})_{(0)}\lhd(T((m\lhd(x\o h))_{(0)})_{(1)}\o x) \o B((m\lhd(x\o h))_{(1)})\\
	&=&(\underbrace{T(m\lhd(x\o h))_{(0)}\lhd (B(T(m\lhd(x\o h))_{(1)3})\o B(T(m\lhd(x\o h))_{(1)1}))})_{(0)}\\
	&&\lhd(\underbrace{T(m\lhd(x\o h))_{(0)}\lhd (B(T(m\lhd(x\o h))_{(1)3})\o B(T(m\lhd(x\o h))_{(1)1}))_{(1)}\o x}) \o T(m\lhd(x\o h))_{(1)2}\\
	&=&(\underbrace{T(m\lhd(x\o h))_{(0)(0)}\lhd (B(T(m\lhd(x\o h))_{(1)3})_{2}\o B(T(m\lhd(x\o h))_{(1)1})_{1})})\\
	&&\lhd(\underbrace{[T(m\lhd(x\o h))_{(0)(1)}, B(T(m\lhd(x\o h))_{(1)3})_{1}, B(T(m\lhd(x\o h))_{(1)1})_{2}]\o x}) \o T(m\lhd(x\o h))_{(1)2}\\
	&=&(T(m\lhd(x\o h))_{(0)}\lhd (B(T(m\lhd(x\o h))_{(1)23})_{2}\o B(T(m\lhd(x\o h))_{(1)21})_{1}))\\
    &&\lhd([T(m\lhd(x\o h))_{(1)1}, B(T(m\lhd(x\o h))_{(1)23})_{1}, B(T(m\lhd(x\o h))_{(1)21})_{2}]\o x) \o T(m\lhd(x\o h))_{(1)22}\\
	&=&T(m\lhd(x\o h))_{(0)}\lhd (B(T(m\lhd(x\o h))_{(1)23})_{2}\o [B(T(m\lhd(x\o h))_{(1)21})_{1},\\
	&&[B(T(m\lhd(x\o h))_{(1)21})_{2}, B(T(m\lhd(x\o h))_{(1)23})_{1}, T(m\lhd(x\o h))_{(1)1}], x]) \o T(m\lhd(x\o h))_{(1)22}\\
  	&=&T(m\lhd(x\o h))_{(0)}\lhd (B(T(m\lhd(x\o h))_{(1)23})_{2}\o [[B(T(m\lhd(x\o h))_{(1)21})_{1},\\
    &&B(T(m\lhd(x\o h))_{(1)21})_{2}, B(T(m\lhd(x\o h))_{(1)23})_{1}], T(m\lhd(x\o h))_{(1)1}, x]) \o T(m\lhd(x\o h))_{(1)22}\\
    &=&T(m\lhd(x\o h))_{(0)}\lhd (B(T(m\lhd(x\o h))_{(1)22})_{2}\o [ B(T(m\lhd(x\o h))_{(1)22})_{1}, T(m\lhd(x\o h))_{(1)1}, x]) \\
    &&\o T(m\lhd(x\o h))_{(1)21}\\
    &=&(T(m\lhd(x\o h))_{(0)}\lhd (B(T(m\lhd(x\o h))_{(1)22})_{2}\o  B(T(m\lhd(x\o h))_{(1)22})_{1}))\lhd (T(m\lhd(x\o h))_{(1)1}\o x) \\
&&\o T(m\lhd(x\o h))_{(1)21}\\
    &=&T(m\lhd(x\o h))_{(0)}\lhd ( T(m\lhd(x\o h))_{(1)1}\o x)\o T(m\lhd(x\o h))_{(1)2}.
\end{eqnarray*}

Furthermore, we get
\begin{eqnarray*}
	\alpha(\widehat{T}(m\o h))
	&=&\alpha(T(m\lhd(x\o h))_{(0)}\lhd ( T(m\lhd(x\o h))_{(1)1}\o x)\o T(m\lhd(x\o h))_{(1)2})\\
	&=&(T(m\lhd(x\o h))_{(0)}\lhd ( T(m\lhd(x\o h))_{(1)1}\o x))\lhd (x\o T(m\lhd(x\o h))_{(1)2})\\
	&=&T(m\lhd(x\o h))_{(0)}\lhd (T(m\lhd(x\o h))_{(1)1}\o [ x,x, T(m\lhd(x\o h))_{(1)2}])\\
	&=&T(m\lhd(x\o h))_{(0)}\lhd (T(m\lhd(x\o h))_{(1)1}\o T(m\lhd(x\o h))_{(1)2})\\
	&=&T(m\lhd(x\o h))_{(0)}\varepsilon(T(m\lhd(x\o h))_{(1)})\\
	&=&T(m\lhd(x\o h))\\
	&=&T(\alpha(m \o h)),				
\end{eqnarray*}
that is, $\alpha\circ\widehat{T}=T\circ\alpha,$ and we can prove that
$$\widehat{T}\circ\beta=\beta\circ \alpha\circ\widehat{T}\circ\beta=\beta\circ T\circ\alpha\circ\beta=\beta\circ T.$$

Hence $(M, T)\cong(M_{x}^{coH}\o H, \widehat{T})$ as Rota-Baxter Hopf heap modules over $(H,B)$.
\end{proof}

\noindent {\bf Acknowledgements}: This work was supported by National Natural Science Foundation of China (12201188).

\end{document}